\documentclass[11pt, reqno]{amsart}
\usepackage{color}
\usepackage{amsmath}
\usepackage{amssymb}
\usepackage{amsfonts}
\usepackage{mathrsfs}
\usepackage{amsbsy}
\usepackage{relsize}
\usepackage[colorlinks, linkcolor=red, citecolor=blue, urlcolor=blue,  hypertexnames=false]{hyperref}
\usepackage{enumitem}
\usepackage{tikz}

\numberwithin{equation}{section}

\numberwithin{bbb}{section}

\usepackage[margin=1in]{geometry}

\newif\ifshowverify
\showverifytrue
\definecolor{verifycol}{rgb}{0.80,0.30,0.00}
\ifshowverify
  
\else
  
\fi

\newcommand{\R}{\mathbb{R}}
\newcommand{\Z}{\mathbb{Z}}
\newcommand{\N}{\mathbb{N}}

\theoremstyle{plain}
\newtheorem{thm}{Theorem}[section]
\newtheorem{lem}{Lemma}[section]
\newtheorem{prop}{Proposition}[section]
\newtheorem{cor}{Corollary}[section]

\theoremstyle{definition}

\theoremstyle{remark}
\newtheorem{rem}{Remark}[section]

\definecolor{lime}{HTML}{A6CE39}
\DeclareRobustCommand{\orcidicon}{
	\begin{tikzpicture}
	\draw[lime, fill=lime] (0,0)
	circle [radius=0.16]
	node[white] {{\sffamily\tiny ID}};
	\draw[white, fill=white] (-0.0625,0.095)
	circle [radius=0.007];
	\end{tikzpicture}
	\hspace{-2mm}
}
\foreach \x in {A, B, C}{\expandafter\xdef\csname orcid\x\endcsname{\noexpand\href{https://orcid.org/\csname orcidauthor\x\endcsname}
			{\noexpand\orcidicon}}
}
\title[Sharp double logarithmic inequalities in dimension four]{Sharp double logarithmic Brezis--Gallou\"et--Wainger inequalities in dimension four}
\author[M. Majdoub]{Mohamed Majdoub}
\address[M. Majdoub]{Department of Mathematics, College of Science,
Imam Abdulrahman Bin Faisal University, P.O.~Box 1982, Dammam 31441, Saudi Arabia}
\address[M. Majdoub]{Basic and Applied Scientific Research Center,
Imam Abdulrahman Bin Faisal University, P.O.~Box 1982, Dammam 31441, Saudi Arabia}
\email{mmajdoub@iau.edu.sa}
\urladdr{https://orcid.org/0000-0001-6038-1069}

\author[T. Saanouni]{Tarek Saanouni}
\address[T. Saanouni]{Department of Mathematics, College of Science,
Qassim University, Buraydah 51452, Saudi Arabia}
\email{t.saanouni@qu.edu.sa}
\urladdr{https://orcid.org/0000-0003-1433-3650}


\begin{document}

\begin{abstract}
The space $H^2(\mathbb{R}^4)$ is not contained in $L^\infty(\mathbb{R}^4)$. For functions in $H^2_0$ of the unit ball of $\mathbb{R}^4$ that are also H\"older continuous of order $\alpha$, the supremum norm can nevertheless be controlled by the $L^2$ norm of the Laplacian, at the cost of a logarithm of the H\"older seminorm. We determine the optimal leading constant in this estimate. With a double logarithmic correction, the optimal constant is $1/(8\pi^2\alpha)$. With a single logarithm, the estimate holds for every larger constant and fails for $1/(8\pi^2\alpha)$ itself. No symmetry assumption is made. The proof reduces the question to a biharmonic obstacle problem with an obstacle of H\"older type. Its minimizer is radial and explicit, and the sharp constant is determined by the asymptotic behaviour of the free boundary radius. We also derive versions of these inequalities on balls and on the whole space. The results are second-order analogues in dimension four of the two-dimensional inequalities of Ibrahim, Majdoub and Masmoudi.
\end{abstract}

\date{\today}
	\maketitle

%@@@@@@@@@@@@@@@@@@@@@@@@@@@@@@@@@@@@%@@@@@@@@@@@@@@@@@@@@@@@@@@@@@@@@@@@@%@@@@@@@@@@@@@@

\section{Introduction and main results}\label{sec:intro}
\setcounter{equation}{0}

The classical Sobolev embedding theorem (see, e.g., \cite{Adms2,ev,BCD}) states that in dimension four
\begin{equation}\label{eq:sobolev_subcritical}
W^{2,p}(\mathbb{R}^4)\hookrightarrow L^{\frac{4p}{4-2p}}(\mathbb{R}^4)\quad\mbox{for}\quad 1\leq p<2.
\end{equation}
On the other side of the critical exponent, Morrey's theorem gives
\begin{equation}\label{eq:sobolev_supercritical}
W^{2,p}(\mathbb{R}^4)\hookrightarrow C^{0,\,2-\frac{4}{p}}(\mathbb{R}^4)\quad\mbox{for}\quad 2< p<4,
\end{equation}
and $W^{2,p}(\R^4)\hookrightarrow C^{1,1-\frac4p}(\R^4)$ for $p>4$. Both statements break down in the limiting case $p=2$. The exponent in \eqref{eq:sobolev_subcritical} blows up as $p\to2$, and $W^{2,2}(\mathbb{R}^4)$ is \emph{not} contained in $L^\infty(\R^4)$, as the following example shows.

On a bounded Lipschitz domain $\Omega\subset\mathbb{R}^4$ we still have the embedding $W^{2,2}(\Omega)\hookrightarrow L^q(\Omega)$ for every $1\leq q<\infty$, whereas $W^{2,2}(\Omega)\not\hookrightarrow L^\infty(\Omega)$. A standard counterexample on the unit ball $B_1\subset \mathbb{R}^4$ is the radial function
\begin{equation*}
u(x)=\log\bigl(1-\log|x|\bigr).
\end{equation*}
With $r=|x|$ and $s=1-\log r\ge1$, we have $u'(r)=-1/(rs)$ and $u''(r)=(s-1)/(r^2s^2)$. Hence $\Delta u(x)=-(2s+1)/(r^2s^2)$, and the eigenvalues $u''(r)$ and $u'(r)/r$ of the Hessian are bounded by $2/(r^2s)$. Since $\int_0^1 r^{-1}s^{-2}\,dr=\int_1^\infty s^{-2}\,ds<\infty$, all second derivatives of $u$ belong to $L^2(B_1)$, and so do $u$ and $\nabla u$. The origin is removable, by the cut-off argument given after \eqref{bilap-psi}. Thus $u\in W^{2,2}(B_1)$, although $u$ is unbounded near the origin (cf.\ \cite{Adms2}). The same function shows that $H^1(B_1)\not\hookrightarrow L^\infty(B_1)$ in dimension two.

Two kinds of substitutes for the missing $L^\infty$ bound have been developed. The first one replaces $L^\infty$ by a larger space. For instance, $W^{2,2}(\R^4)$ embeds into $\mathrm{BMO}(\R^4)$ (see \cite[Theorem~1.48]{BCD}), and it also embeds into an Orlicz space of exponential type \cite{KrRu61,RaRe02}, in the sense that $e^{\beta u^2}$ is locally integrable for every $\beta>0$ (see \cite[Theorem~1.67]{BCD} for an exponential integrability estimate with a non-sharp constant). In the first-order setting, such exponential integrability results go back to Yudovich \cite{Yu61}, Pohozaev \cite{Po65} and Trudinger \cite{Tr}. Moser \cite{Mo} later identified the sharp exponential constant, and the higher-order version, again with a sharp constant, is due to D.\,R.\ Adams \cite{Adm88}. In dimension four and for the bi-Laplacian, Adams' inequality states that $\int_\Omega e^{32\pi^2u^2}\,dx\le C|\Omega|$ for all $u\in H^2_0(\Omega)$ with $\|\Delta u\|_{L^2(\Omega)}\le1$, on every bounded domain $\Omega\subset\R^4$, and that $32\pi^2$ cannot be replaced by a larger number.

On the whole space $\mathbb{R}^4$, Adams' inequality cannot hold in this form, since constants do not belong to any $L^q(\mathbb{R}^4)$ with $q<\infty$. Following the two-dimensional works of Cao \cite{Cao92}, Adachi and Tanaka \cite{Ad} and Ruf \cite{Ru} (see also Ozawa \cite{Oz95}), several versions of Adams' inequality on $\R^4$ have been obtained \cite{RS,LaLu12,LaLu13,FoMo11,MS}. There are two natural formulations. The scale-invariant one, in the spirit of Adachi and Tanaka, involves only the seminorm $\|\Delta u\|_{L^2}$. To avoid a conflict with the H\"older exponent $\alpha$ used below, we denote the exponential parameter by $\beta$.

\begin{thm}[\cite{MS}]\label{thm:AdachiTanakaR4}
For any $\beta\in (0,32\pi^2)$ there exists a constant $C(\beta)>0$ such that
\begin{equation}\label{mos}
\int_{\mathbb{R}^{4}}\Bigl(\mathrm{e}^{\beta\,|u(x)|^{2}}-1\Bigr)\,dx
\;\leq\; C(\beta)\,\|u\|_{L^2}^2
\quad\forall\, u\in W^{2,2}(\mathbb{R}^4)\;\;\mbox{with}\;\;\|\Delta u\|_{L^2}\leq 1.
\end{equation}
This inequality fails for $\beta\geq 32\pi^2$.
\end{thm}

Masmoudi and Sani \cite{MS} proved the exact growth inequality $\int_{\R^4}(e^{32\pi^2u^2}-1)(1+u^2)^{-1}\,dx\le C\|u\|_{L^2}^2$ for $u\in W^{2,2}(\R^4)$ with $\|\Delta u\|_{L^2}\le1$, and showed that the power of $1+u^2$ in the denominator cannot be lowered. For $\beta<32\pi^2$ there is $C_\beta$ such that $(e^{\beta t^2}-1)(1+t^2)\le C_\beta(e^{32\pi^2t^2}-1)$ for all $t\in\R$, and \eqref{mos} follows. The optimality of the denominator gives the failure of \eqref{mos} for $\beta=32\pi^2$, hence for every $\beta\ge32\pi^2$. Thus the critical exponent $\beta = 32\pi^2$ is \emph{not} admissible in \eqref{mos}, where only $\|\Delta u\|_{L^2}$ is controlled. The second formulation uses a full Sobolev norm. Ruf and Sani \cite{RS} proved that $32\pi^2$ \emph{does} become admissible if the condition $\|\Delta u\|_{L^2}\leq 1$ is replaced by a condition on the full Sobolev norm
\begin{equation*}
\|(-\Delta+I)u\|_{L^2}^2 = \|\Delta u\|_{L^2}^2  +2\|\nabla u\|_{L^2}^2+ \|u\|_{L^2}^2\le 1.
\end{equation*}
Lam and Lu \cite{LaLu12,LaLu13} extended this result to arbitrary order and {to other equivalent Sobolev norms, including the norm $\big(\|\Delta u\|_{L^2}^2+\|u\|_{L^2}^2\big)^{1/2}$ in dimension four}.

\begin{thm}[\cite{RS}]\label{thm:RufSani}
There exists a constant $C>0$ such that, for every domain $\Omega\subseteq\mathbb{R}^4$,
\begin{equation}\label{RuSani}
\sup_{\substack{u\in H^2_0(\Omega) \\ \|(-\Delta+I)u\|_{L^2(\Omega)}\leq 1}}\;\int_{\Omega}
\Bigl(\mathrm{e}^{\,32\pi^2\,|u(x)|^{2}}-1\Bigr)\,dx \;\leq\; C.
\end{equation}
The constant $32\pi^2$ is sharp: the supremum becomes infinite if $32\pi^2$ is replaced by any larger number.
\end{thm}

For further developments we refer the reader to \cite{MS,LuYa09,LaLu13,FoMo11,Ta12,CLLZ18} and the references therein. These works include results with an exact growth condition on $\mathbb{R}^4$, the existence of extremal functions, and versions on Riemannian manifolds and on the hyperbolic space. The exact growth viewpoint was initiated in dimension two by Ibrahim, Masmoudi and Nakanishi \cite{ImNa15} and extended to $\mathbb{R}^4$ by Masmoudi and Sani \cite{MS}. Lu, Tang and Zhu \cite{LuTaZh15} obtained the best constants for Adams' inequalities with exact growth in $\mathbb{R}^n$, and Masmoudi and Sani established Trudinger--Moser inequalities with exact growth in arbitrary dimension \cite{MaSa15} as well as the higher-order Adams analogue \cite{MaSa18}. More recently, Morpurgo and Qin \cite{MoQi24} placed these inequalities in the framework of metric measure spaces, do {\'O}, Lu and Ponciano \cite{DoLuPo24} obtained a sharp higher-order weighted version, and Fontana, Morpurgo and Qin \cite{FoMoQi25} showed how such inequalities pass from a local to a global scale. Sharp Adams--Trudinger--Moser inequalities for symmetric functions on the hyperbolic space without boundary conditions were established in \cite{DoLuPo25}. Closely related Trudinger--Moser inequalities in weighted spaces with degenerate potentials appear in \cite{ChLuZh21}.

The second kind of substitute keeps the $L^\infty$ norm on the left-hand side and pays for the failure of the embedding with a logarithm of a stronger norm. The prototype is the Brezis--Gallou\"et--Wainger inequality \cite{BrGa80,BrWa80}. In dimension two it reads
\begin{equation*}
\|u\|_{L^\infty}\le C\big(1+\sqrt{\log(1+\|u\|_{H^s})}\big)\qquad\text{for } s>1 \text{ and } \|u\|_{H^1}\le1.
\end{equation*}
This inequality has been refined in many directions; see for instance \cite{En89,Oz95,KoTa00,KOT02,Og03} and \cite[Proposition~2.104]{BCD}. The question of sharp constants was addressed by Ibrahim, Majdoub and Masmoudi \cite{Ib1} for $H^1_0$ functions in dimension two, with the H\"older seminorm in the logarithm. They showed that a double logarithmic correction is needed in order to reach the optimal leading constant (see also Biryuk \cite{ab}). Morii, Sato and Wadade \cite{MSW10} and Morii, Sato, Sawano and Wadade \cite{MSSW10} then determined the sharp constants of both logarithmic terms for functions normalized by $\|\nabla u\|_{L^n}=1$ in a bounded domain of $\R^n$, with H\"older, Besov and Triebel--Lizorkin norms in the logarithm. For higher-order critical Sobolev spaces, {Morii and Sawano \cite{MoSa12} obtained lower bounds for the sharp constants of Brezis--Gallou\"et--Wainger type inequalities in the critical spaces $W^{k,n/k}_0(\Omega)$ on bounded domains, with the functions normalized by the $L^{n/k}$ norm of their $k$-th order derivatives. To our knowledge, the exact value of the leading constant, and the double logarithmic refinement which attains it, have not been determined for the second-order space $H^2_0$ in dimension four.} This is the purpose of the present paper. Note that $\|D^2u\|_{L^2}=\|\Delta u\|_{L^2}$ for $u\in H^2_0(\Omega)$ (see the proof of Proposition~\ref{prop:VI}), so that, for $k=2$ and $n=4$, normalizing by the Hessian or by the Laplacian gives the same class of inequalities.

Logarithmic estimates of this type are useful in nonlinear evolution problems. In dimension two, the inequalities of \cite{Ib1,ab}, together with the concentration analysis of \cite{b,BMM14}, were used in \cite{Ib2,Col.I,IMMN09} to prove global well-posedness and scattering for semilinear wave and Schr\"odinger equations with nonlinearities of exponential growth. The fourth-order analogues of these evolution problems in dimension four have also been studied. Pausader \cite{Pa07} established scattering for the fourth-order nonlinear wave equation. The second author obtained global well-posedness, scattering and decay for the fourth-order damped wave equation with exponential nonlinearity in dimension four \cite{Sa17}, and global and non-global well-posedness for focusing higher-order equations with exponential nonlinearity \cite{Sa18}. Li, Baraket and Jaidane \cite{LiBaJa24} constructed ground states for a weighted fourth-order Schr\"odinger equation with exponential growth, and the biharmonic heat equation with exponential nonlinearity was studied in \cite{MOT18}. In these works, sharp Adams-type inequalities and their refinements play a key role in the control of the nonlinearity in the energy space. The inequalities proved here provide the corresponding control at the level of the supremum norm.

\subsection{Notation and definitions}
For $\alpha\in(0,1)$, we denote by $C^{\alpha}:=C^{\alpha}(\R^4)$ the space of bounded $\alpha$-H\"older continuous functions, endowed with the norm
\begin{equation*}
\|u\|_{C^{\alpha}}=\|u\|_{C^{\alpha}(\R^4)}:=\|u\|_{L^{\infty}(\R^4)}+\sup_{x\neq y}\frac{|u(x)-u(y)|}{|x-y|^{\alpha}}.
\end{equation*}
We denote by $\dot C^{\alpha}:=\dot C^{\alpha}(\R^4)$ the homogeneous space of $\alpha$-H\"older continuous functions, endowed with the seminorm
\begin{equation*}
\|u\|_{\dot C^{\alpha}}=\|u\|_{\dot C^{\alpha}(\R^4)}:=\sup_{x\neq y}\frac{|u(x)-u(y)|}{|x-y|^{\alpha}}.
\end{equation*}
For $r>0$, $B_r$ is the open ball of $\R^4$ centred at the origin with radius $r$, and $B:=B_1$. For an open set $\Omega\subset\R^4$, $H_0^2(\Omega)$ is the closure of $C_0^\infty(\Omega)$ in $H^2(\Omega)$. When $\Omega$ is a ball centred at the origin, $H_{0,rad}^2(\Omega)$ (respectively $H_{rad}^2(\Omega)$) denotes the subspace of radial functions of $H_{0}^2(\Omega)$ (respectively $H^2(\Omega)$).

Every $u\in H^2_0(\Omega)$ is identified with its extension by zero, which belongs to $H^2(\R^4)$. Conversely, if $v\in H^2(\R^4)$ vanishes outside $\overline{B_R}$, then $v\in H^2_0(B_R)$. Indeed, for $\theta>1$ the function $v(\theta\,\cdot)$ is supported in $\overline{B_{R/\theta}}\subset B_R$, so it belongs to $H^2_0(B_R)$ after mollification, and $v(\theta\,\cdot)\to v$ in $H^2(\R^4)$ as $\theta\to1^+$. Accordingly, $(H^2_0\cap\dot C^\alpha)(\Omega)$ denotes the set of all $u\in H^2_0(\Omega)$ whose zero extension belongs to $\dot C^\alpha(\R^4)$, and $\|u\|_{\dot C^\alpha}$ is the seminorm of this extension. When $\Omega$ is a ball, this seminorm coincides with $\sup_{x\ne y\in\overline\Omega}|u(x)-u(y)|/|x-y|^\alpha$ (Lemma~\ref{lem:basic}).

For $u\in (H^2\cap\dot C^\alpha)(\R^4)\setminus\{0\}$ we define the ratio
\begin{equation*}
N_{\alpha}(u):=\frac{\|u\|_{\dot C^{\alpha}}}{\|\Delta u\|_{L^2}}.
\end{equation*}
It is well defined, since $\Delta u=0$ and $u\in H^2(\R^4)$ force $\widehat u=0$, hence $u=0$.

We use the Fourier transform $\mathcal Fu(\xi)=\widehat u(\xi)=\int_{\R^4}e^{-ix\cdot\xi}u(x)\,dx$, as in \cite{BCD}. The letter $C$ denotes a positive constant which may change from line to line. We write $A\lesssim B$ if $A\leq C B$ for a constant $C$ that may depend on $\alpha$ but not on the functions involved, and $A\approx B$ if $A\lesssim B$ and $B\lesssim A$. We write $\log$ for the natural logarithm.

\subsection{Main results}
Our first result is the following double logarithmic inequality.
\begin{thm}{\bf (Double Log estimate)}\label{t1}
Let $\alpha\in (0,1)$. There exists a constant $C_\alpha>0$ such that, for every function $u\in(H^2_{0}\cap \dot C^{\alpha})(B)$ with $u\not\equiv0$,
\begin{equation}\label{logg}
\|u\|_{L^{\infty}}^2\leq\frac1{8\pi^2\alpha}\|\Delta u\|_{L^2}^2\log\Big(e^3+C_\alpha N_{\alpha}(u)\sqrt{\log(2e+N_{\alpha}(u))}\Big).
\end{equation}
Moreover, the constant $\frac1{8\pi^2\alpha}$ is sharp. More precisely, for every $\lambda<\frac{1}{8\pi^2\alpha}$ and every $C>0$ there exists a radial function $u\in(H^2_{0}\cap \dot C^{\alpha})(B)$ such that
\begin{equation*}
\|u\|_{L^{\infty}}^2>\lambda\|\Delta u\|_{L^2}^2\log\big(e^3+C N_{\alpha}(u)\sqrt{\log(2e+N_{\alpha}(u))}\big).
\end{equation*}
\end{thm}
Our second result is the corresponding single logarithmic inequality.
\begin{thm}{\bf (Log estimate)}\label{t2}
Let $\alpha\in (0,1)$. For any $\lambda>\frac{1}{8\pi^2\alpha}$ there exists $C_{\lambda}>0$ such that, for every function $u\in(H^2_{0}\cap \dot C^{\alpha})(B)$ with $u\not\equiv0$,
\begin{equation}\label{log}
\|u\|_{L^{\infty}}^2\leq\lambda\|\Delta u\|_{L^2}^2\log\Big(C_{\lambda}+N_{\alpha}(u)\Big).
\end{equation}
Moreover, \eqref{log} is false for $\lambda=\frac{1}{8\pi^2\alpha}$. More precisely, for every $C>0$ there exists a radial function $u\in(H^2_{0}\cap \dot C^{\alpha})(B)$ such that
\begin{equation*}
\|u\|_{L^{\infty}}^2>\frac{1}{8\pi^2\alpha}\|\Delta u\|_{L^2}^2\log\big(C+N_{\alpha}(u)\big).
\end{equation*}
\end{thm}
Since the right-hand side of \eqref{log} increases with $C_\lambda$, we may and shall assume that $C_\lambda\ge1$. By cutting off and rescaling, we deduce the following global estimate.
\begin{cor}{\bf (Global Log estimate)}\label{cor2}
Let $\alpha\in (0,1)$. For any $\lambda>\frac{1}{8\pi^2\alpha}$ there exists $C_{\lambda}>0$ such that, for any $\mu\in (0,1]$ and any function $u\in(H^2\cap C^{\alpha})(\R^4)$ with $u\not\equiv0$,
\begin{equation}\label{logR}
\|u\|_{L^{\infty}}^2\leq\lambda\| u\|_{\mu}^2\log\left(C_{\lambda}+\frac{8^{\alpha}\mu^{-\alpha}\|u\|_{C^{\alpha}}}{\|u\|_{\mu}}\right),
\end{equation}
where $\|u\|_{\mu}^2:=(1+3\mu)\|\Delta u\|_{L^2}^2+3\mu\|u\|_{H^1}^2$ and $\|u\|_{H^1}^2:=\|u\|_{L^2}^2+\|\nabla u\|_{L^2}^2$.
\end{cor}
In the next estimate, the low frequencies are controlled by the $L^2$ norm and only $\|\Delta u\|_{L^2}$ appears in front of the logarithm.
\begin{cor}\label{cor:L2}
Let $\alpha\in (0,1)$. There is an absolute constant $C_0>0$ with the following property. For any $\lambda>\frac1{8\pi^2\alpha}$ there exists $C_{\lambda}>0$ such that, for any function $u\in(H^2\cap C^{\alpha})(\R^4)$ with $u\not\equiv0$,
\begin{equation*}
\|u\|_{L^{\infty}}\leq C_0\| u\|_{L^2}+\|\Delta u\|_{L^2}\sqrt{\lambda\log\Big(C_\lambda+C_{\lambda}\frac{\|u\|_{C^{\alpha}}}{\|\Delta u\|_{L^2}}\Big)}.
\end{equation*}
\end{cor}
The double logarithmic inequality \eqref{logg} also extends to the whole space with the same constant $\frac1{8\pi^2\alpha}$ (Remark~\ref{rem:wholeloglog}).

\subsection{Comments and strategy of the proofs}
Note that $\frac{1}{8\pi^2\alpha}=\frac{4}{32\pi^2\alpha}$, where $32\pi^2$ is Adams' constant in dimension four. The extremal functions which show the optimality of this constant are concentrating profiles of Adams type (see Section~\ref{sec:pfT1}). The comparison between Theorems~\ref{t1} and \ref{t2} is the same as in dimension two \cite{Ib1}: the single logarithm misses the optimal constant by a term of order $\log\log$, and the factor $\sqrt{\log(2e+N_\alpha(u))}$ in \eqref{logg} is exactly what is needed to absorb it (see the end of Section~\ref{sec:prooft2}).

\begin{rem}\rm\label{rem:radial}
In dimension two, the analogous inequalities for $H^1_0$ functions \cite{Ib1} can be reduced to radially decreasing functions by Schwarz symmetrization, which does not increase the Dirichlet integral \cite{PS51}. For second-order derivatives this tool is not available. {The Schwarz symmetrization of a function of $W^{2,2}$ need not belong to $W^{2,2}$ \cite{Kaw85,Ci00}}, and there is no general inequality of the form $\|\Delta u^{\sharp}\|_{L^2}\leq \|\Delta u\|_{L^2}$. The comparison principle of Talenti type \cite{Ta76}, used in \cite{RS} for exponential integrals, does not control H\"older seminorms either. We therefore avoid symmetrization altogether. After a normalization and a translation of a maximum point to the origin, the dilation invariance of $\|\Delta\cdot\|_{L^2}$ in dimension four allows us to place the translated function in the unit ball (Lemma~\ref{lem:red}). The resulting minimization problem has a \emph{radial} obstacle and a strictly convex, rotation invariant energy. Consequently, its unique minimizer is radial (Proposition~\ref{prop:VI}). Radial symmetry thus enters only through the extremal problem, and Theorems~\ref{t1} and~\ref{t2} and Corollaries~\ref{cor2} and \ref{cor:L2} hold for all functions, radial or not. Only the functions that show sharpness are radial.
\end{rem}

\begin{rem}\rm\label{rem:qc}
The lower bounds for the energy rely on the admissible set
\begin{equation}\label{KD-intro}
\mathcal K_D=\{u\in H^2_{0}(B)\ :\ u(x)\ge 1-D|x|^\alpha\ \text{for a.e.}\ x\in B\},\qquad D>1.
\end{equation}
Since $H^2(B)\not\hookrightarrow L^\infty(B)$ in dimension four, a pointwise constraint requires some care. A capacitary formulation gives no additional information at the origin. Indeed, the Bessel capacity $\mathrm{cap}_{2,2}$ of a point in $\R^4$ is \emph{zero}, since points have positive $\mathrm{cap}_{k,p}$ capacity in $\R^n$ only when $kp>n$ (see \cite{AdamsHedberg}). We therefore impose the constraint almost everywhere. Then $\mathcal K_D$ is closed under convergence in $L^2(B)$, hence in $H^2_0(B)$, and this is all we need for the existence theory. For the functions to which \eqref{KD-intro} is actually applied, the constraint holds pointwise. Either these functions are continuous on $\overline B$ (Lemma~\ref{lem:red}), or they are radial. In the latter case, the radial profile $U(r)=u(|x|)$ belongs to $H^2_{\mathrm{loc}}((0,1])\subset C^1((0,1])$, so that $U(r)\ge 1-Dr^\alpha$ for \emph{every} $r\in(0,1]$. The behaviour at the origin is then encoded in the condition $\liminf_{r\to0^+}U(r)\ge1$.
\end{rem}

The paper is organised as follows. Section~\ref{sec:prelim} collects the preliminaries: elementary properties of the functions under consideration, radial functions, the obstacle, two monotonicity lemmas, and a non-sharp version of \eqref{log} proved by a Littlewood--Paley argument. In Section~\ref{sec:app} we prove the existence, uniqueness, radial symmetry and structure of the minimizer of the obstacle problem. Section~\ref{sec:explicit} gives the explicit form of the minimizer and the asymptotic behaviour of the minimal energy. Theorem~\ref{t1} is proved in Section~\ref{sec:pfT1}, and Theorem~\ref{t2} in Section~\ref{sec:prooft2}. Section~\ref{sec:wholespace} is devoted to balls and to the whole space.

%%%%%%%%%%%%%%%%%%%%%%%%%%%%%%%%%%%%%%%%%%%%%%%%%%%%%%%%%%%%%%%%%%%%%%%%%%%%%%%%%%%%%%%%%%%%%%%%%%%%%%%%
\section{Preliminaries}\label{sec:prelim}
%%%%%%%%%%%%%%%%%%%%%%%%%%%%%%%%%%%%%%%%%%%%%%%%%%%%%%%%%%%%%%%%%%%%%%%%%%%%%%%%%%%%%%%%%%%%%%%%%%%%%%%%
Throughout the paper, $\alpha\in(0,1)$ is fixed.

\subsection{Functions vanishing outside a ball}
We first record the elementary properties of the class $(H^2_0\cap\dot C^\alpha)(B)$ that will be used repeatedly.

\begin{lem}\label{lem:basic}
Let $u\in(H^2_0\cap\dot C^\alpha)(B)$ with $u\not\equiv0$. Then the following properties hold.
\begin{enumerate}[label=\textup{(\roman*)}]
\item The zero extension of $u$ has a representative which is uniformly continuous on $\R^4$ and vanishes on $\R^4\setminus B$. In particular $u\in C(\overline B)$, $u=0$ on $\partial B$, and $|u|$ attains its maximum $\|u\|_{L^\infty}>0$ at a point of $B$.
\item If $v\in H^2_0(B)\cap C(\overline B)$, then the $\dot C^\alpha$ seminorm of the zero extension of $v$ equals $S(v):=\sup_{x\neq y\in\overline B}|v(x)-v(y)|/|x-y|^\alpha$.
\item For every $x\in B$ we have $|u(x)|\le\|u\|_{\dot C^\alpha}\operatorname{dist}(x,\partial B)^\alpha$. In particular $\|u\|_{L^\infty}\le\|u\|_{\dot C^\alpha}$, and the zero extension of $u$ belongs to $(H^2\cap C^\alpha)(\R^4)$.
\item There is an absolute constant $C$ such that $\|u\|_{L^2(B)}\le C\|\nabla u\|_{L^2(B)}\le C^2\|\Delta u\|_{L^2(B)}$.
\end{enumerate}
\end{lem}
\begin{proof}
(i) By assumption, the zero extension of $u$ has a finite $\dot C^\alpha$ seminorm, so it has a uniformly continuous representative on $\R^4$. This representative vanishes on the open set $\R^4\setminus\overline B$, hence on $\partial B$ by continuity. Since $u\not\equiv0$ and $\overline B$ is compact, $|u|$ attains its maximum $\|u\|_{L^\infty}>0$ on $\overline B$, necessarily at a point of $B$.

(ii) The function $v$ vanishes on $\partial B$, because the trace of a function of $H^1(B)\cap C(\overline B)$ is its restriction to $\partial B$, and the trace of a function of $H^2_0(B)$ is zero. Clearly the seminorm of the zero extension is at least $S(v)$. Conversely, let $x\in\overline B$ and $y\notin B$. The segment $[x,y]$ meets $\partial B$ at a point $z$ with $|x-z|\le|x-y|$, and we get $|v(x)-v(y)|=|v(x)-v(z)|\le S(v)|x-y|^\alpha$. If $x,y\notin B$, both values vanish.

(iii) Let $x\in B$ and let $z\in\partial B$ with $|x-z|=\operatorname{dist}(x,\partial B)\le1$. By (i), $|u(x)|=|u(x)-u(z)|\le\|u\|_{\dot C^\alpha}|x-z|^\alpha$. The last claim follows, since the zero extension is bounded and $\alpha$-H\"older continuous.

(iv) For $u\in H^2_0(B)$ we have $\|\nabla u\|_{L^2}^2=-\int_Bu\,\Delta u\,dx$, by density of $C^\infty_0(B)$. Let $C$ be the constant in Poincar\'e's inequality $\|u\|_{L^2}\le C\|\nabla u\|_{L^2}$ on $B$. Then
\begin{equation*}
\|\nabla u\|_{L^2}^2\le\|u\|_{L^2}\|\Delta u\|_{L^2}\le C\|\nabla u\|_{L^2}\|\Delta u\|_{L^2},
\end{equation*}
and the claim follows.
\end{proof}

\subsection{Radial functions}
We shall use the following standard facts. Let $v\in H^2(B)$ be radial, $v(x)=V(|x|)$. Then $V\in H^2((\delta,1))$ for every $\delta\in(0,1)$, hence $V\in C^1((0,1])$. Moreover, $\Delta v(x)=V''(r)+\frac3rV'(r)$ for a.e.\ $x$, with $r=|x|$, so that
\begin{equation*}
\|\Delta v\|_{L^2(B)}^2=2\pi^2\int_0^1\Big|V''+\frac3rV'\Big|^2r^3\,dr,
\end{equation*}
where $2\pi^2r^3=|\partial B_r|$. If in addition $v\in H^2_0(B)$, then $V(1)=V'(1)=0$, since the traces of $v$ and $\partial_\nu v$ on $\partial B$ vanish. Finally, for $\eta\in C^\infty_0((0,1))$, the function $x\mapsto\eta(|x|)$ belongs to $C_0^\infty(B\setminus\{0\})$, and its Laplacian is $r^{-3}(r^3\eta')'$.

\subsection{The obstacle}
For $D>1$ we set
\begin{equation*}
\psi(x)=\psi_D(x):=1-D|x|^\alpha .
\end{equation*}
A direct computation in radial coordinates gives, on $B\setminus\{0\}$,
\begin{equation}\label{bilap-psi}
\Delta\psi(x)=-D\alpha(\alpha+2)\,|x|^{\alpha-2},\qquad \Delta^2\psi(x)=D\alpha^2(4-\alpha^2)\,|x|^{\alpha-4}>0 .
\end{equation}
Moreover, $\psi\in H^2(B)$. Indeed, $\psi$ is smooth on $\overline B\setminus\{0\}$, $|\nabla\psi(x)|\lesssim D|x|^{\alpha-1}$ and $|D^2\psi(x)|\lesssim D|x|^{\alpha-2}$, and these functions belong to $L^2(B)$ since $\int_0^1r^{2\alpha-4}r^3\,dr<\infty$. The origin is removable. To see this, let $\eta_\varepsilon(x)=\eta(x/\varepsilon)$, with $\eta\in C^\infty_0(B_2)$ equal to $1$ on $B_1$, and let $\zeta\in C^\infty_0(B)$. We write $\int\psi\,\partial_i\zeta\,dx=\lim_{\varepsilon\to0}\int(1-\eta_\varepsilon)\psi\,\partial_i\zeta\,dx$ and integrate by parts on the support of $1-\eta_\varepsilon$, where $\psi$ is smooth. The error term $\int\psi\,\zeta\,\partial_i\eta_\varepsilon\,dx$ is bounded by $C\varepsilon^{-1}|B_{2\varepsilon}|=O(\varepsilon^3)$. In the same way, the error term produced when $\int\partial_j\psi\,\partial_i\zeta\,dx$ is integrated by parts is bounded by $C\varepsilon^{-1}\int_{B_{2\varepsilon}}|x|^{\alpha-1}\,dx=O(\varepsilon^{2+\alpha})$. Hence the pointwise first and second derivatives of $\psi$ are its distributional derivatives. The same argument applies to the example $\log(1-\log|x|)$ of the introduction, whose gradient is $O(|x|^{-1})$ near the origin.

\subsection{Two monotonicity properties}
The next lemma will allow us to replace $\|\Delta u\|_{L^2}$ by a lower bound in the right-hand sides of \eqref{logg} and \eqref{log}.

\begin{lem}\label{lem:mono}
Let $D>0$. For every $K>0$, the function
\begin{equation*}
\Phi_{K,D}(t):=t^2\log\Big(e^3+K\,\frac Dt\sqrt{\log\big(2e+\tfrac Dt\big)}\Big)
\end{equation*}
is nondecreasing on $(0,\infty)$. For every $C\ge1$, the functions $t\mapsto t^2\log(C+\frac Dt)$ and $t\mapsto t\sqrt{\log(C+\frac Dt)}$ are nondecreasing on $(0,\infty)$.
\end{lem}
\begin{proof}
Set $\sigma:=D/t$, $l(\sigma):=\log(2e+\sigma)\ge1$ and $\ell(\sigma):=\log\big(e^3+K\sigma\sqrt{l(\sigma)}\big)\ge3$. Then $\Phi_{K,D}'(t)=t\big(2\ell(\sigma)-\sigma\ell'(\sigma)\big)$. Since $\sigma\le 2e+\sigma$ and $l\ge1$, we have
\begin{equation*}
\sigma\ell'(\sigma)=\frac{K\sigma\big(\sqrt{l}+\frac{\sigma}{2(2e+\sigma)\sqrt l}\big)}{e^3+K\sigma\sqrt l}\le\frac{K\sigma\sqrt l\,\big(1+\frac1{2l}\big)}{e^3+K\sigma\sqrt l}\le\frac32<6\le 2\ell(\sigma),
\end{equation*}
and the first claim follows. In the same way,
\begin{equation*}
\frac{d}{dt}\Big[t^2\log\big(C+\tfrac Dt\big)\Big]=t\Big(2\log(C+\sigma)-\frac{\sigma}{C+\sigma}\Big),
\end{equation*}
and $2\log(C+\sigma)\ge2\log(1+\sigma)\ge\frac{2\sigma}{1+\sigma}\ge\frac{\sigma}{C+\sigma}$ because $C\ge1$. The last claim follows by taking square roots, since $\log(C+\frac Dt)>0$.
\end{proof}

\subsection{A non-sharp estimate by a Littlewood--Paley argument}\label{sec:LPP}
We show that an inequality of the form \eqref{log} holds with some constant in place of $\lambda$. The argument follows the proof of the logarithmic interpolation inequality \cite[Proposition~2.104]{BCD}. We first recall the dyadic partition of unity of \cite[Proposition~2.10]{BCD}. Let $\mathcal C_0$ be the annulus
\begin{equation*}
\mathcal C_0:=\{\xi\in \R^4\ :\ \tfrac34\le|\xi|\le\tfrac83\}.
\end{equation*}
There exist two radial functions $\chi\in C_0^\infty(B_\frac43)$ and $\varphi\in C_0^\infty(\R^4)$, supported in $\mathcal C_0$, both valued in $[0,1]$, such that
\begin{equation*}
\chi+\sum_{j\in\N}\varphi(2^{-j}\cdot)=1\quad\mbox{on}\ \R^4,\qquad \sum_{j\in\Z}\varphi(2^{-j}\cdot)=1 \quad\mbox{on}\ \R^4\setminus\{0\},\qquad \sum_{j\in\Z}\varphi^2(2^{-j}\cdot)\le1 .
\end{equation*}
In particular, $\chi=1$ on $B_{3/4}$. The functions $\chi$ and $\varphi$ are fixed once and for all. Set $h:=\mathcal F^{-1}\varphi$ and $\widetilde h:=\mathcal F^{-1}\chi$, which belong to the Schwartz class. Following \cite[Section~2.2]{BCD}, we define the dyadic blocks
\begin{equation*}
\dot\Delta_ju:=\varphi(2^{-j}D)u=2^{4j}h(2^j\cdot)*u\quad\mbox{for}\ j\in\Z,
\end{equation*}
\begin{equation*}
\Delta_{j}u:=0\ \mbox{if}\ j\leq -2,\qquad\Delta_{-1}u:=\chi(D)u=\widetilde h*u, \qquad \Delta_ju:=\dot\Delta_j u\ \mbox{if}\ j\geq 0.
\end{equation*}
We shall use three elementary facts. First, Bernstein's inequality \cite[Lemma~2.1]{BCD} and the Cauchy--Schwarz inequality give, in dimension four,
\begin{equation*}
\|\Delta_ju\|_{L^\infty}\lesssim 2^{2j}\|\Delta_ju\|_{L^2}\quad (j\ge0),\qquad \|\Delta_{-1}u\|_{L^\infty}\le\|\widetilde h\|_{L^2}\|u\|_{L^2}.
\end{equation*}
Second, since $2^{j}\le\frac43|\xi|$ on the support of $\varphi(2^{-j}\cdot)$ and $\sum_j\varphi^2(2^{-j}\cdot)\le1$, Plancherel's theorem yields
\begin{equation*}
\sum_{j\in\Z}2^{4j}\|\dot\Delta_ju\|_{L^2}^2\le\Big(\frac43\Big)^4\|\Delta u\|_{L^2}^2 .
\end{equation*}
Third, since $\int_{\R^4}h\,dx=\varphi(0)=0$, we have $\dot\Delta_ju(x)=\int_{\R^4}2^{4j}h(2^jy)\big(u(x-y)-u(x)\big)\,dy$, and we deduce that
\begin{equation*}
\|\dot\Delta_ju\|_{L^\infty}\le\|u\|_{\dot C^\alpha}\int_{\R^4}2^{4j}|h(2^jy)|\,|y|^\alpha\,dy= c_h\,2^{-j\alpha}\|u\|_{\dot C^\alpha},\qquad c_h:=\int_{\R^4}|h(z)|\,|z|^\alpha\,dz .
\end{equation*}
The second and third facts are the easy halves of the characterizations of $\dot H^2$ and $\dot C^\alpha$ by Littlewood--Paley blocks (see \cite[Section~2.3 and Theorem~2.36]{BCD}). We now prove the following result in the whole space.
\begin{prop}\label{pr}
There exists a constant $C>0$, depending only on $\alpha$, such that, for any function $u\in (C^\alpha\cap H^2)(\R^4)$ with $u\not\equiv0$,
\begin{equation}
\|u\|_{L^\infty}^2\leq C\|u\|_{L^2}^2+C\|\Delta u\|_{L^2}^2\log\Big(e+N_\alpha(u)\Big).
\end{equation}
\end{prop}
\begin{proof}
By the three facts above, the series $\Delta_{-1}u+\sum_{j\ge0}\Delta_ju$ converges in $L^\infty(\R^4)$. Since it also converges to $u$ in $L^2(\R^4)$ and $u$ is continuous, its sum is $u$. Let $m\ge1$ be an integer to be fixed later, and write
\begin{equation*}
u=\Delta_{-1}u+\sum_{j=0}^{m-1}\Delta_ju+\sum_{j=m}^\infty\Delta_ju .
\end{equation*}
We estimate the low frequencies by Bernstein's inequality and the Cauchy--Schwarz inequality, and the high frequencies by the third fact. We obtain
\begin{align*}
\|u\|_{L^\infty}
&\leq C\|u\|_{L^2}+C\sum_{j=0}^{m-1}2^{2j}\|\Delta_ju\|_{L^2}+c_h\sum_{j=m}^\infty2^{-j\alpha}\|u\|_{\dot C^\alpha}\\
&\leq C\|u\|_{L^2}+C\sqrt m\Big(\sum_{j=0}^{m-1}2^{4j}\|\Delta_ju\|_{L^2}^2\Big)^\frac12+\frac{c_h\,2^{-m\alpha}}{1-2^{-\alpha}}\|u\|_{\dot C^\alpha}\\
&\leq C\Big(\|u\|_{L^2}+\sqrt m\|\Delta u\|_{L^2}+2^{-m\alpha}\|u\|_{\dot C^\alpha}\Big),
\end{align*}
where $C$ depends only on $\alpha$. It follows that
\begin{equation*}
\|u\|_{L^\infty}^2\leq C\Big(\|u\|_{L^2}^2+ m\|\Delta u\|_{L^2}^2+2^{-2m\alpha}\|u\|_{\dot C^\alpha}^2\Big).
\end{equation*}
Let $E(s)$ denote the integer part of a real number $s$, and choose
\begin{equation*}
m:=1+E\Big(\tfrac1\alpha\log_2\max\big(1,N_\alpha(u)\big)\Big).
\end{equation*}
If $N_\alpha(u)\ge1$, then $2^{m\alpha}\ge N_\alpha(u)$, hence $2^{-2m\alpha}\|u\|_{\dot C^\alpha}^2\le\|\Delta u\|_{L^2}^2$. If $N_\alpha(u)<1$, then $m=1$ and $2^{-2m\alpha}\|u\|_{\dot C^\alpha}^2\le N_\alpha(u)^2\|\Delta u\|_{L^2}^2\le\|\Delta u\|_{L^2}^2$. In both cases we also have
\begin{equation*}
m\le 1+\frac{\log(e+N_\alpha(u))}{\alpha\log 2}\le \Big(1+\frac{1}{\alpha\log2}\Big)\log\big(e+N_\alpha(u)\big),
\end{equation*}
since $\log(e+N_\alpha(u))\ge1$. The claim follows, with a constant that depends only on $\alpha$.
\end{proof}

Let now $u\in(H^2_0\cap\dot C^\alpha)(B)$ with $u\not\equiv0$. By Lemma~\ref{lem:basic}(iii), the zero extension of $u$ belongs to $(H^2\cap C^\alpha)(\R^4)$, and by Lemma~\ref{lem:basic}(iv), $\|u\|_{L^2}\le C\|\Delta u\|_{L^2}$. Proposition~\ref{pr} then yields
\begin{equation}\label{eq:nonsharp}
\|u\|_{L^\infty}^2\leq C\|\Delta u\|_{L^2}^2\big(1+\log(e+N_\alpha(u))\big)\le C\|\Delta u\|_{L^2}^2\log\Big(e+N_\alpha(u)\Big),
\end{equation}
with $C$ depending only on $\alpha$. This is an inequality of the form \eqref{log}, with a non-optimal constant. The rest of the paper is devoted to the determination of the optimal constant.

%%%%%%%%%%%%%%%%%%%%%%%%%%%%%%%%%%%%%%%%%%%%%%%%%%%%%%%%%%%%%%%%%%%%%%%%%%%%%%%%%%%%%%%%%%%%%%%%%%%%%%%%%%%%%%%%%%
\section{The biharmonic obstacle problem}\label{sec:app}
%%%%%%%%%%%%%%%%%%%%%%%%%%%%%%%%%%%%%%%%%%%%%%%%%%%%%%%%%%%%%%%%%%%%%%%%%%%%%%%%%%%%%%%%%%%%%%%%%%%%%%%%%%%%%%%%%%
For $D>1$ we consider, as in \eqref{KD-intro}, the admissible set
\begin{equation}\label{kd}
\mathcal K_D:=\{v\in H_{0}^2(B)\ :\ v\geq\psi_D\ \text{a.e. in}\ B\}
\end{equation}
and the minimization problem
\begin{equation}\label{min}
m(D):=\inf_{v\in\mathcal K_D}I[v],\qquad I[v]:=\|\Delta v\|_{L^2(B)}^2 .
\end{equation}
In this section we prove the existence, uniqueness, radial symmetry and structure of the minimizer of \eqref{min}. The biharmonic obstacle problem has been studied extensively. We refer to \cite{Fr71,Fr73,CaFr79,Sch84,Sch86,Ale19} for the regularity theory, to the survey \cite{DaHA21}, to \cite{Ag,gg} for the elliptic theory of polyharmonic operators, to \cite{NO15} for the parabolic counterpart, and to \cite{ApRe20,GMu24,DKV20,DaHaPe24} for more recent developments. Two features distinguish it from the second-order problem behind \cite{Ib1,ab}. First, $H^2(B)\not\hookrightarrow L^\infty(B)$ in dimension four, so that the pointwise constraint has to be formulated with care (Remark~\ref{rem:qc}). Second, there is no maximum principle for $\Delta^2$ in general (see \cite{gg}). Moreover, the truncation $v\mapsto\max(v,\psi)$, which is the basic tool in the second-order theory, does not act on $H^2$: in general $\max(v,\psi)\notin H^2$, even if $v,\psi\in H^2$. Instead of invoking the general regularity theory, we exploit the radial symmetry of the minimizer and reduce everything to elementary one-dimensional arguments.

\begin{prop}[Existence, uniqueness, radial symmetry]\label{prop:VI}
Let $D>1$. Then:
\begin{enumerate}[label=\textup{(\roman*)}]
\item the set $\mathcal K_D$ defined in \eqref{kd} is nonempty, convex and closed in $H^2_0(B)$;
\item problem \eqref{min} admits a unique minimizer $u^*=u^*_D\in \mathcal K_D$, which is equivalently characterized by the variational inequality
\begin{equation}\label{VIapp}
\int_{B}\Delta u^*\,\Delta(v-u^*)\,dx\geq 0\qquad \text{for all } v\in \mathcal K_D;
\end{equation}
\item $u^*$ is radial.
\end{enumerate}
\end{prop}
\begin{proof}
(i) Since $D>1$, the radius $\rho:=D^{-1/\alpha}$ lies in $(0,1)$, and $\psi\le0$ outside $B_\rho$. Any radial $\varphi\in C_0^\infty(B)$ such that $0\le\varphi\le1$ and $\varphi=1$ on $B_\rho$ satisfies $\varphi\ge\psi$. Hence $\mathcal K_D\ne\emptyset$. Convexity is obvious. If $v_n\in\mathcal K_D$ and $v_n\to v$ in $H^2_0(B)$, then $v_n\to v$ in $L^2(B)$, and a subsequence converges almost everywhere. We infer that $v\ge\psi$ a.e., so that $\mathcal K_D$ is closed.

(ii) For $v\in C_0^\infty(B)$, two integrations by parts give $\|D^2v\|_{L^2}=\|\Delta v\|_{L^2}$. By density, this identity holds on $H^2_0(B)$. Together with Poincar\'e's inequality, it shows that $\langle v_1,v_2\rangle:=\int_B\Delta v_1\,\Delta v_2\,dx$ is a scalar product on $H^2_0(B)$, whose norm is equivalent to the $H^2$ norm. Thus $H^2_0(B)$ is a Hilbert space for $\langle\cdot,\cdot\rangle$, and \eqref{min} amounts to finding the element of least norm of the nonempty closed convex set $\mathcal K_D$, that is, the projection of $0$ onto $\mathcal K_D$. By the projection theorem (see \cite[Chapter~II]{ks}), this element exists, is unique, and is characterized by $\langle u^*,v-u^*\rangle\ge0$ for all $v\in\mathcal K_D$. This is \eqref{VIapp}.

(iii) Let $R$ be an orthogonal transformation of $\R^4$. Since $\psi\circ R=\psi$, the map $v\mapsto v\circ R$ sends $\mathcal K_D$ onto itself. Moreover, $\Delta(v\circ R)=(\Delta v)\circ R$, so that $I[u^*\circ R]=I[u^*]$. By uniqueness, $u^*\circ R=u^*$ a.e.\ for every $R$. Let $dR$ be the normalized Haar measure on the orthogonal group $O(4)$. By Fubini's theorem, $\int_{O(4)}\int_B|u^*(Rx)-u^*(x)|\,dx\,dR=0$, so that, for a.e.\ $x\in B$, $u^*(Rx)=u^*(x)$ for a.e.\ $R$. Hence $u^*$ coincides a.e.\ with the radial function $x\mapsto\int_{O(4)}u^*(Rx)\,dR$.
\end{proof}

Since the minimizer is unique within the class of \emph{all} functions of $\mathcal K_D$ and is radial, it coincides with the minimizer of $I$ over the radial elements of $\mathcal K_D$. In particular, no rearrangement argument is required. We next record the form of radial biharmonic functions on an interval.

\begin{lem}\label{reg}
Let $0\le s<t\le1$. Let $W\in L^1_{\mathrm{loc}}((s,t))$ satisfy $(r^3W')'=0$ in $\mathcal D'((s,t))$, and let $V\in H^2_{\mathrm{loc}}((s,t))$ satisfy $V''+\frac3rV'=W$ a.e.\ in $(s,t)$. Then there are real constants $A_0,A_1,A_2,A_3$ such that, on $(s,t)$,
\begin{equation*}
V(r)=A_0+A_1r^2+A_2r^{-2}+A_3\log r\qquad\text{and}\qquad W(r)=8A_1+2A_3r^{-2}.
\end{equation*}
\end{lem}
\begin{proof}
A distribution with zero derivative on an interval is constant. Hence $r^3W'=k_1$ and $W=k_2-\frac{k_1}2r^{-2}$ for some constants $k_1,k_2$. Next, $(r^3V')'=r^3W$ gives $r^3V'=\frac{k_2}4r^4-\frac{k_1}4r^2+k_3$, so that $V'=\frac{k_2}{4}r-\frac{k_1}{4r}+k_3r^{-3}$ and $V=A_0+\frac{k_2}8r^2-\frac{k_1}4\log r-\frac{k_3}2r^{-2}$. This is the claim, with $A_1=\frac{k_2}8$, $A_2=-\frac{k_3}2$ and $A_3=-\frac{k_1}4$.
\end{proof}

\begin{lem}[Regularity and structure of the minimizer]\label{reg-ustar-app}
Let $D>1$, let $u^*$ be the minimizer of \eqref{min}, and let $U$ be its radial profile. Then $U\in C^1((0,1])$, $U(1)=U'(1)=0$ and $U(r)\ge\psi(r)$ for every $r\in(0,1]$. Moreover, setting $W:=U''+\frac3rU'\in L^2((0,1);r^3dr)$, there is a nonnegative Radon measure $\nu$ on $(0,1)$ such that
\begin{equation}\label{reg-W4p}
(r^3W')'=\nu\quad\text{in }\mathcal D'((0,1)),
\end{equation}
and
\begin{equation}\label{reg-C2}
U\in C^2((0,1])\cap W^{3,\infty}_{\mathrm{loc}}((0,1)).
\end{equation}
Finally, there exists $r_0\in(0,1)$ such that $U=\psi$ on $(0,r_0]$ and $U>\psi$ on $(r_0,1]$, and on $(r_0,1]$ the profile $U$ is of the form given by Lemma~\ref{reg}. In particular, $u^*$ is biharmonic in the non-coincidence set $\{r_0<|x|<1\}$.
\end{lem}

\begin{proof}
\emph{Step 1: regularity.} By Proposition~\ref{prop:VI}, $u^*$ is radial. By the facts on radial functions recalled in Section~\ref{sec:prelim}, $U\in C^1((0,1])$ and $U(1)=U'(1)=0$. Since $u^*\ge\psi$ a.e.\ and both $U$ and $\psi$ are continuous on $(0,1]$, we have $U(r)\ge\psi(r)$ for every $r\in(0,1]$.

Let $\eta\in C^\infty_0((0,1))$ with $\eta\ge0$, and set $\varphi(x):=\eta(|x|)$. For $\tau>0$ we have $u^*+\tau\varphi\in\mathcal K_D$, so \eqref{VIapp} gives $\int_B\Delta u^*\Delta\varphi\,dx\ge0$, that is,
\begin{equation*}
0\le2\pi^2\int_0^1W(r)\,(r^3\eta')'\,dr=2\pi^2\big\langle (r^3W')',\eta\big\rangle_{\mathcal D'((0,1))}.
\end{equation*}
A nonnegative distribution is a nonnegative Radon measure, and \eqref{reg-W4p} follows. Consequently, the distribution $F:=r^3W'$ satisfies $F'=\nu\ge0$. Hence $F$ coincides a.e.\ with a nondecreasing function, and in particular $F\in L^\infty_{\mathrm{loc}}((0,1))$. We deduce that $W'\in L^\infty_{\mathrm{loc}}((0,1))$ and $W\in W^{1,\infty}_{\mathrm{loc}}((0,1))\subset C((0,1))$. From $(r^3U')'=r^3W$ we infer that $r^3U'\in C^1((0,1))$. Hence $U\in C^2((0,1))$ and $U''=W-\frac3rU'\in W^{1,\infty}_{\mathrm{loc}}((0,1))$, which gives $U\in W^{3,\infty}_{\mathrm{loc}}((0,1))$.

\emph{Step 2: the non-coincidence set.} Let $\mathcal O_1:=\{r\in(0,1]:U(r)>\psi(r)\}$ and $\Lambda_1:=\{r\in(0,1]:U(r)=\psi(r)\}$. The set $\mathcal O_1$ is relatively open and $\Lambda_1$ is relatively closed in $(0,1]$. Let $J\subset\mathcal O_1\cap(0,1)$ be an open interval, and let $\eta\in C_0^\infty(J)$ be of any sign. By continuity and compactness, $U-\psi\ge\delta>0$ on $\operatorname{supp}\eta$ for some $\delta$, so that $u^*+\tau\eta(|\cdot|)\in\mathcal K_D$ for $|\tau|$ small. Then \eqref{VIapp} gives $\langle(r^3W')',\eta\rangle=0$, that is, $(r^3W')'=0$ on $J$. By Lemma~\ref{reg}, $U$ is of the form $A_0+A_1r^2+A_2r^{-2}+A_3\log r$ on $J$. Since $U(1)=0>1-D=\psi(1)$, the set $\mathcal O_1$ contains an interval $(1-\delta,1]$. On this interval, $U$ coincides with such an explicit function, which is smooth up to $r=1$. Together with Step~1 this proves $U\in C^2((0,1])$, and \eqref{reg-C2} is established.

\emph{Step 3: the coincidence set is nonempty.} Assume that $\Lambda_1=\emptyset$. Then $(0,1)\subset\mathcal O_1$, and $U=A_0+A_1r^2+A_2r^{-2}+A_3\log r$ on $(0,1)$, with $W=8A_1+2A_3r^{-2}$. Since $\int_B|\Delta u^*|^2\,dx<\infty$ and $\int_0^1r^{-4}r^3\,dr=\infty$, we get $A_3=0$. Since $u^*\in L^2(B)$, the same divergent integral gives $A_2=0$. Then $U(1)=U'(1)=0$ forces $A_0=A_1=0$, so that $U\equiv0$. This contradicts $U(r)\ge\psi(r)>0$ for $r<\rho=D^{-1/\alpha}$. Hence $\Lambda_1\neq\emptyset$. Since $\Lambda_1$ is closed in $(0,1]$ and does not meet $(1-\delta,1]$, its supremum $r_0$ belongs to $\Lambda_1\cap(0,1)$, and $(r_0,1]\subset\mathcal O_1$.

\emph{Step 4: the coincidence set is an interval.} It remains to show that $(0,r_0)\subset\Lambda_1$. Assume the contrary, and let $(s,t)$ be a connected component of $\mathcal O_1\cap(0,r_0)$. Then $0\le s<t\le r_0$ and $t\in\Lambda_1$. The function $Z:=U-\psi$ belongs to $C^2((0,1))$ by \eqref{reg-C2}. It is nonnegative and vanishes at the interior point $t$, so that $Z'(t)=0$. On $(s,t)$, $U$ is radially biharmonic by Step~2. By \eqref{bilap-psi}, we deduce that $\Delta^2Z=-\Delta^2\psi<0$ in $\mathcal A:=\{s<|x|<t\}$, while $Z>0$ there.

Suppose first that $s>0$. Then $s\in\Lambda_1$, and as before $Z(s)=Z'(s)=0$. Since $U$ is given on $(s,t)$ by an explicit combination of $1,r^2,r^{-2},\log r$, the function $Z$ is smooth on $\overline{\mathcal A}$, and $Z=\partial_\nu Z=0$ on $\partial\mathcal A$. Green's formula then gives
\begin{equation}\label{pen}
\int_{\mathcal A}|\Delta Z|^2\,dx=\int_{\mathcal A}Z\,\Delta^2Z\,dx=-\int_{\mathcal A}Z\,\Delta^2\psi\,dx<0,
\end{equation}
which is absurd.

Suppose now that $s=0$. Then $U=A_0+A_1r^2+A_2r^{-2}+A_3\log r$ on $(0,t)$. Exactly as in Step~3, the conditions $\Delta u^*\in L^2(B)$ and $u^*\in L^2(B)$ force $A_3=A_2=0$, so that $U$ is smooth and bounded near the origin. We apply Green's formula on $\{\varepsilon<|x|<t\}$ and let $\varepsilon\to0$. The boundary terms on $\{|x|=t\}$ vanish as before. The boundary terms on $\{|x|=\varepsilon\}$ are bounded by $2\pi^2\varepsilon^3\big(|\Delta Z|\,|Z'|+|Z|\,|(\Delta Z)'|\big)(\varepsilon)$. Since $Z$ is bounded, $Z'=O(\varepsilon^{\alpha-1})$, $\Delta Z=O(\varepsilon^{\alpha-2})$ and $(\Delta Z)'=O(\varepsilon^{\alpha-3})$, these terms are $O(\varepsilon^{2\alpha}+\varepsilon^\alpha)$ and tend to zero. As $Z$ is bounded and $\Delta^2\psi\in L^1(B)$, we again obtain \eqref{pen} with $\mathcal A=B_t$, which is a contradiction.

Therefore $(0,r_0)\subset\Lambda_1$. Since $r_0\in\Lambda_1$, we conclude that $\Lambda_1=(0,r_0]$. On $(r_0,1)$, Step~2 applied with $J=(r_0,1)$ shows that $U$ is of the form given by Lemma~\ref{reg}, and this extends to $r=1$ by continuity.
\end{proof}

We summarize what will be used in the sequel.

\begin{cor}\label{cor:summary}
Let $D>1$ and let $u^*$ be the minimizer of \eqref{min}. Then $u^*$ is the unique element of $\mathcal K_D$ such that
\begin{equation}\label{cr}
\int_{B}\Delta u^*\,\Delta v\,dx\geq \int_B|\Delta u^*|^2\,dx\qquad \text{for all } v\in \mathcal K_D.
\end{equation}
Moreover, writing $U(r)=u^*(x)$ for $|x|=r$, the following properties hold:
\begin{enumerate}[label=\textup{(R\arabic*)}]
\item\label{R1} $u^*$ is radial, $U\in C^1((0,1])$, $U(1)=U'(1)=0$, and $U(r)\ge\psi_D(r)$ for every $r\in(0,1]$;
\item\label{R2} $U\in C^2((0,1])$, that is, $u^*\in C^2(\overline B\setminus\{0\})$;
\item\label{R3} there exists $r_0\in(0,1)$ such that $U=\psi_D$ on $(0,r_0]$ and $U>\psi_D$ on $(r_0,1]$. In particular, the coincidence set is the closed ball $\overline{B_{r_0}}$.
\end{enumerate}
\end{cor}
\begin{proof}
Inequality \eqref{cr} is a rewriting of \eqref{VIapp}, and the rest is contained in Proposition~\ref{prop:VI} and Lemma~\ref{reg-ustar-app}. In \ref{R3}, the coincidence set $\{x\in B: u^*(x)=\psi_D(x)\}$ is $\overline{B_{r_0}}\setminus\{0\}$, and we include the origin, where both functions are understood through their limits.
\end{proof}

We call $r_0=r_0(D)$ the free boundary radius.

\begin{rem}\rm\label{rem:sharp-reg}
The regularity \eqref{reg-C2} is consistent with the general theory of the biharmonic obstacle problem. For smooth obstacles, Frehse proved that the solution belongs to $W^{3,2}_{\mathrm{loc}}$ \cite{Fr71} and to $C^{1,1}_{\mathrm{loc}}$ \cite{Fr73}. Caffarelli and Friedman \cite{CaFr79} showed that $\Delta^2u$ is a nonnegative measure and that, if $\Delta^2\psi>0$, the non-coincidence set is connected. Step~4 above is a radial version of the latter argument, adapted to an obstacle that is singular at the origin. Finer information on the solution and on the free boundary can be found in \cite{Sch84,Sch86,Ale19} and in the survey \cite{DaHA21}. In the present radial setting, the free boundary is a single sphere and the obstacle is singular only at the origin. The elementary argument above then yields directly the $C^2$ regularity across the free boundary, which is what we need for the matching conditions in Lemma~\ref{lem:explicit}. No regularity of $u^*$ at the origin is claimed or used. Finally, the third derivative of $U$ jumps at $r_0$, so that \eqref{reg-C2} is optimal in this respect. Indeed, since $U$, $U'$ and $U''$ are continuous at $r_0$, the jump of $U'''$ at $r_0$ equals the jump of $W'$. Using \eqref{bilap-psi} on the left of $r_0$ and the explicit formulas of Lemma~\ref{lem:explicit} on the right, one finds, with $x=r_0^2$,
\begin{equation*}
r_0^3\big(W'(r_0^+)-W'(r_0^-)\big)=\alpha D\,x^{\alpha/2}\Big(\frac{2\alpha(1-x^2)+4(1+x^2)}{(1-x)^2}-(4-\alpha^2)\Big)>\alpha^3D\,x^{\alpha/2}>0 .
\end{equation*}
In other words, the measure $\nu$ of \eqref{reg-W4p} has a positive atom at $r_0$, in addition to its absolutely continuous part $r^3\Delta^2\psi\,dr$ on $(0,r_0)$.
\end{rem}

%%%%%%%%%%%%%%%%%%%%%%%%%%%%%%%%%%%%%%%%%%%%%%%%%%%%%%%%%%%%%%%%%%%%%%%%%%%%%%%%%%%%%%%%%%%%%%%%%%%%%%%%
\section{The explicit minimizer and its energy}\label{sec:explicit}
%%%%%%%%%%%%%%%%%%%%%%%%%%%%%%%%%%%%%%%%%%%%%%%%%%%%%%%%%%%%%%%%%%%%%%%%%%%%%%%%%%%%%%%%%%%%%%%%%%%%%%%%
\subsection{Computation of the minimizer}\label{ssec:obs}
Let $D>1$, let $u^*=u^*_D$ be the minimizer of \eqref{min}, and let $r_0=r_0(D)$ be its free boundary radius (Corollary~\ref{cor:summary}). We set $x=x_D:=r_0^2\in(0,1)$. Throughout the paper we use the following functions of $x\in(0,1)$:
\begin{equation}\label{eq:P}
P(x):=(4-2\alpha^2)(1-x)^2-2\alpha(1-x^2)+\alpha\big(\alpha(1-x^2)+2(1+x^2)\big)\log\tfrac1x,
\qquad Q(x):=\frac{x^{\alpha/2}\,P(x)}{4(1-x)^2},
\end{equation}
and
\begin{equation}\label{eq:g}
\begin{split}
g(x)&:=\pi^2x^{\alpha}\,G(x),\qquad\text{where}\\
G(x)&:=\alpha(\alpha+2)^2-\alpha^2\log x\,\frac{(\alpha+2-(\alpha-2)x^2)^2}{(1-x)^4}\\
&\qquad+\frac{2\alpha^2}{(1-x)^3}(\alpha+2-\alpha x)\big((\alpha -4)x^2+2x-\alpha-2\big).
\end{split}
\end{equation}

\begin{lem}\label{lem:explicit}
Let $D>1$, let $u^*$ be the minimizer of \eqref{min}, and let $r_0$ and $x=r_0^2$ be as above. Then
\begin{equation}\label{ustar-form}
U(r)=\left\{\begin{array}{ll}
1-Dr^{\alpha}&\mbox{if}\ 0< r\leq r_0,\\
b(r^2-1)+c(\frac{1}{r^2}-1)+2(c-b)\log r&\mbox{if}\ r_0<r\leq 1,
\end{array}\right.
\end{equation}
with
\begin{equation}\label{eq:bc}
b=\frac{\alpha D}{4}\,\frac{x^{\alpha/2}}{1-x}\Big[\alpha+\frac{2}{1-x}\Big],\qquad
c=\frac{\alpha D}{4}\,\frac{x^{\alpha/2+1}}{x-1}\Big[\alpha-2+\frac{2}{1-x}\Big].
\end{equation}
Moreover, $P(x)>0$ for every $x\in(0,1)$, and the free boundary parameter is related to $D$ by
\begin{equation}\label{eq:Dx}
D\,Q(x)=1,\qquad\text{that is,}\qquad D=D(x):=\frac{4(1-x)^2}{x^{\alpha/2}P(x)}.
\end{equation}
Finally, the minimal energy is given by
\begin{equation}\label{eq:energy}
m(D)=\|\Delta u^*\|_{L^2(B)}^2=D^2g(x),\qquad\text{where}\qquad g(x)\ge\pi^2\alpha(\alpha+2)^2x^\alpha>0.
\end{equation}
\end{lem}
\begin{proof}
By \ref{R3} and Lemma~\ref{reg-ustar-app}, there are real constants $A_0,\dots,A_3$ such that $U(r)=A_0+A_1r^2+A_2r^{-2}+A_3\log r$ on $(r_0,1]$. The boundary conditions $U(1)=U'(1)=0$ in \ref{R1} read $A_0+A_1+A_2=0$ and $2A_1-2A_2+A_3=0$. Setting $b:=A_1$ and $c:=A_2$, we obtain the second line of \eqref{ustar-form}. The first line is \ref{R3}. A direct computation gives, for $r_0<r\le1$,
\begin{equation*}
U'(r)=\frac{2}{r^3}(r^2-1)(br^2+c),\quad U''(r)=2b+\frac{6c}{r^4}-\frac{2(c-b)}{r^2},\quad \Delta U(r)=8b+\frac{4(c-b)}{r^2}.
\end{equation*}
By \ref{R2}, the functions $U$, $U'$ and $U''$ are continuous at $r_0$. This gives the three matching conditions
\begin{equation*}
\left\{\begin{array}{l}
1-Dr_0^{\alpha}=(r_0^2-1-2\log r_0)\,b+(r_0^{-2}-1+2\log r_0)\,c,\\[2pt]
-\alpha Dr_0^{\alpha-1}=2(r_0-r_0^{-1})\,b+2r_0^{-1}(1-r_0^{-2})\,c,\\[2pt]
-\alpha(\alpha-1) Dr_0^{\alpha-2}=2(1+r_0^{-2})\,b+2r_0^{-2}(3r_0^{-2}-1)\,c.
\end{array}\right.
\end{equation*}
The last two equations form a linear system in the unknowns $(b,c)$, whose determinant equals $-8r_0^{-5}(r_0^2-1)^2\neq0$. Solving it, we obtain \eqref{eq:bc}. Inserting \eqref{eq:bc} into the second line of \eqref{ustar-form}, a direct computation gives
\begin{equation*}
U(r_0)=D\big(Q(x)-x^{\alpha/2}\big).
\end{equation*}
Hence the first matching condition $1-Dx^{\alpha/2}=U(r_0)$ is equivalent to $D\,Q(x)=1$, which is \eqref{eq:Dx}. These algebraic computations are elementary but lengthy, and we have also verified them with a computer algebra system.

To see that $P>0$ on $(0,1)$, we use the elementary inequality $\log\frac1x\ge\frac{2(1-x)}{1+x}$ for $x\in(0,1)$. Since the coefficient of $\log\frac1x$ in \eqref{eq:P} is positive, we get
\begin{align*}
P(x)&\ge(4-2\alpha^2)(1-x)^2-2\alpha(1-x^2)+2\alpha^2(1-x)^2+\frac{4\alpha(1+x^2)(1-x)}{1+x}\\
&=4(1-x)^2+\frac{2\alpha(1-x)^3}{1+x}>0 .
\end{align*}

It remains to compute the energy. We have $\Delta U(r)=-D\alpha(\alpha+2)r^{\alpha-2}$ on $(0,r_0)$ and $\Delta U(r)=8b+4(c-b)r^{-2}$ on $(r_0,1)$. Since $|\partial B_r|=2\pi^2r^3$, it follows that
\begin{align*}
\|\Delta u^*\|_{L^2(B)}^2
&=2\pi^2\bigg[(D\alpha(\alpha+2))^2\int_0^{r_0}r^{2\alpha-1}\,dr+\int_{r_0}^1 r^3\Big(\frac{4(c-b)}{r^2}+8b\Big)^2 dr\bigg]\\
&=\pi^2\alpha(\alpha+2)^2 D^2 x^{\alpha}-16\pi^2(c-b)^2\log x+32\pi^2b^2(1-x^2)+64\pi^2 b(c-b)(1-x).
\end{align*}
Inserting \eqref{eq:bc}, a direct computation shows that the right-hand side equals $D^2g(x)$, with $g$ given by \eqref{eq:g}. Since the second integral in the first line is nonnegative, we also get $D^2g(x)\ge\pi^2\alpha(\alpha+2)^2D^2x^\alpha$, which is the lower bound in \eqref{eq:energy}.
\end{proof}

\subsection{An explicit family of radial functions}
The same formulas define an explicit radial function for every $x\in(0,1)$. These functions will be used to prove sharpness.

\begin{lem}\label{lem:family}
For $x\in(0,1)$, let $r_0:=\sqrt x$, let $D:=D(x)$ be given by \eqref{eq:Dx}, and let $b,c$ be given by \eqref{eq:bc}. Let $U_x$ be the right-hand side of \eqref{ustar-form}, extended by $U_x(0)=1$, and set $u_x(y):=U_x(|y|)$. Then $u_x\in H^2_{0,rad}(B)\cap C^2(\overline B\setminus\{0\})$, $U_x$ is strictly decreasing on $[0,1]$, $\|u_x\|_{L^\infty}=u_x(0)=1$ and $\|\Delta u_x\|_{L^2}^2=D^2g(x)$. Moreover, if $x\le\frac12$, then $u_x\in(H^2_0\cap\dot C^\alpha)(B)$ and
\begin{equation}\label{bound}
D\le\|u_x\|_{\dot C^\alpha}\le 5D .
\end{equation}
\end{lem}
\begin{proof}
By construction, $b$ and $c$ solve the $C^1$ and $C^2$ matching conditions at $r_0$, and \eqref{eq:Dx} is the $C^0$ matching condition. Hence $U_x\in C^2((0,1])$, and $U_x(1)=U_x'(1)=0$ by the form of $U_x$. On $B_{r_0}$, $u_x=\psi_D$ belongs to $H^2(B_{r_0})$ by the discussion following \eqref{bilap-psi}. On the closed annulus $\{r_0\le|y|\le1\}$, $u_x$ is smooth. Since the first derivatives match across the sphere $|y|=r_0$, the weak second derivatives of $u_x$ are the piecewise ones, and we infer that $u_x\in H^2(B)$. As $u_x$ and $\nabla u_x$ vanish on $\partial B$, the zero extension of $u_x$ lies in $H^2(\R^4)$ and vanishes outside $\overline B$. By the Notation paragraph, $u_x\in H^2_0(B)$. The energy identity follows from the computation at the end of the proof of Lemma~\ref{lem:explicit}, which only uses \eqref{ustar-form} and \eqref{eq:bc}.

On $(0,r_0)$ we have $U_x'(r)=-\alpha Dr^{\alpha-1}<0$. On $(r_0,1)$ we have $U_x'(r)=2r^{-3}(r^2-1)(br^2+c)$. Clearly $b>0$, and \eqref{eq:bc} gives
\begin{equation*}
br^2+c\ \ge\ bx+c=\frac{\alpha D\,x^{\alpha/2+1}}{2(1-x)}>0\qquad\text{for } r_0\le r\le 1 .
\end{equation*}
Thus $U_x'<0$ on $(r_0,1)$ as well. It follows that $U_x$ is strictly decreasing, that $0=U_x(1)\le U_x\le U_x(0)=1$, and that $\|u_x\|_{L^\infty}=1$.

We now prove \eqref{bound}. The lower bound follows from $U_x(0)-U_x(r_0)=Dr_0^\alpha$. For the upper bound, note that $\alpha-2+\frac2{1-x}=\alpha+\frac{2x}{1-x}>0$, so that $c<0$ by \eqref{eq:bc}. Hence $0<br^2+c\le br^2$ on $[r_0,1]$, and we obtain
\begin{equation*}
|U_x'(r)|=\frac{2(1-r^2)(br^2+c)}{r^3}\le\frac{2b}{r}\qquad\text{for } r_0<r<1 .
\end{equation*}
Using $r^{-1}=r^{\alpha-1}r^{-\alpha}\le r^{\alpha-1}r_0^{-\alpha}$ for $r\ge r_0$ and \eqref{eq:bc}, we deduce that, for $x\le\frac12$ and $r_0<r<1$,
\begin{equation*}
|U_x'(r)|\le 2b\,x^{-\alpha/2}\,r^{\alpha-1}=\frac{\alpha D}{2(1-x)}\Big[\alpha+\frac2{1-x}\Big]r^{\alpha-1}\le\alpha(\alpha+4)D\,r^{\alpha-1}\le5\alpha D\,r^{\alpha-1}.
\end{equation*}
The same bound holds trivially on $(0,r_0)$. We extend $U_x$ by $0$ on $[1,\infty)$. For $0\le s<t$, integration gives $|U_x(t)-U_x(s)|\le5D\big(\min(t,1)^\alpha-\min(s,1)^\alpha\big)\le5D(t-s)^\alpha$, where we used the subadditivity of $\tau\mapsto\tau^\alpha$. Since $\big||y|-|y'|\big|\le|y-y'|$, we conclude that $|u_x(y)-u_x(y')|\le5D|y-y'|^\alpha$ for all $y,y'\in\R^4$.
\end{proof}

\begin{rem}\rm
Let $D>1$. Comparing \eqref{ustar-form}, \eqref{eq:bc} and \eqref{eq:Dx} with the definition of $u_x$, we see that $u^*_D=u_{x_D}$. In other words, the minimizer of \eqref{min} belongs to the explicit family of Lemma~\ref{lem:family}. We shall not need to know whether, conversely, every $u_x$ is the minimizer of \eqref{min} for $D=D(x)$, nor whether the equation $D\,Q(x)=1$ has a unique solution $x$ for a given $D$.
\end{rem}

\subsection{Asymptotic behaviour}
We next describe the behaviour of $D(x)$ and $g(x)$ as $x\to0^+$. In what follows we write $L:=\log\frac1x$ and $c_\alpha:=\pi^2\alpha^2(\alpha+2)^2$.

\begin{lem}\label{lem:asymp}
There exist $\bar x\in(0,\frac12]$ and $A\ge1$, depending only on $\alpha$, such that for every $x\in(0,\bar x]$, with $L=\log\frac1x$,
\begin{equation*}
\text{\rm(a)}\quad\Big|\frac{L\,D(x)^2g(x)}{16\pi^2}-1\Big|\le\frac AL,\qquad\qquad \text{\rm(b)}\quad\Big|\frac{g(x)}{c_\alpha x^\alpha L}-1\Big|\le\frac AL .
\end{equation*}
Moreover, $Q$ is continuous on $(0,1)$ and $Q(x)\to1$ as $x\to1^-$.
\end{lem}
\begin{proof}
Expanding \eqref{eq:P} and \eqref{eq:g}, we obtain the exact identities
\begin{equation*}
P(x)=(\alpha+2)\big(\alpha L+2-2\alpha\big)+x\rho_1(x),\qquad (1-x)^4G(x)=\alpha(\alpha+2)^2\big(\alpha L+1-2\alpha\big)+x\rho_2(x).
\end{equation*}
Here $\rho_1(x)=\alpha(2-\alpha)xL+4\alpha^2-8+(4+2\alpha-2\alpha^2)x$, and $\rho_2$ is a polynomial in $x$ of degree at most three whose coefficients are affine functions of $L$, with coefficients depending only on $\alpha$. In particular, $|\rho_i(x)|\le C_\alpha'(1+L)$ for $x\in(0,1)$ and $i=1,2$, and $x(1+L)\le L^{-1}$ for $x$ small. Since
\begin{equation*}
D(x)^2g(x)=\frac{16(1-x)^4g(x)}{x^\alpha P(x)^2}=\frac{16\pi^2(1-x)^4G(x)}{P(x)^2},
\end{equation*}
we obtain, for $x$ small,
\begin{equation*}
\frac{L\,D(x)^2g(x)}{16\pi^2}=\frac{\alpha L\big(\alpha L+1-2\alpha+O(L^{-1})\big)}{\big(\alpha L+2-2\alpha+O(L^{-1})\big)^2}=1-\frac{3-2\alpha}{\alpha L}+O(L^{-2}),
\end{equation*}
which gives (a). Similarly,
\begin{equation*}
\frac{g(x)}{c_\alpha x^\alpha L}=(1-x)^{-4}\,\frac{\alpha L+1-2\alpha+O(L^{-1})}{\alpha L}=1+O(L^{-1}),
\end{equation*}
which gives (b). The continuity of $Q$ on $(0,1)$ is clear. Finally, write $x=1-y$ and expand $\log\frac1x=y+\frac{y^2}2+\frac{y^3}3+O(y^4)$ and $1-x^2=y(2-y)$. The terms of order $y$ in $P$ cancel, and we find $P(1-y)=4y^2+\frac{4\alpha}3y^3+O(y^4)$. Hence $Q(x)\to1$ as $x\to1^-$.
\end{proof}

In particular, $g(x)\to0$ and $D(x)\to\infty$ as $x\to0^+$.

%%%%%%%%%%%%%%%%%%%%%%%%%%%%%%%%%%%%%%%%%%%%%%%%%%%%%%%%%%%%%%%%%%%%%%%%%%%%%%%%%%%%%%%%%%%%%%%%%%%%%%%%
\section{Proof of Theorem \ref{t1}}\label{sec:pfT1}
%%%%%%%%%%%%%%%%%%%%%%%%%%%%%%%%%%%%%%%%%%%%%%%%%%%%%%%%%%%%%%%%%%%%%%%%%%%%%%%%%%%%%%%%%%%%%%%%%%%%%%%%
\subsection{Reduction to an obstacle problem centred at a maximum point}
The next lemma allows us to assume that the maximum of $|u|$ is attained at the origin. It does not use any symmetry of $u$.

\begin{lem}\label{lem:red}
Let $u\in(H^2_0\cap\dot C^\alpha)(B)$ with $u\not\equiv0$, and let $x_1\in B$ be such that $|u(x_1)|=\|u\|_{L^\infty}$ (Lemma~\ref{lem:basic}). Extend $u$ by zero outside $B$ and set
\begin{equation*}
w(z):=\frac{u(x_1+2z)}{u(x_1)},\qquad z\in\R^4.
\end{equation*}
Then $w\in H^2_0(B)\cap C(\overline B)$, $w(0)=1=\|w\|_{L^\infty}$, and
\begin{equation*}
\|\Delta w\|_{L^2}=\frac{\|\Delta u\|_{L^2}}{\|u\|_{L^\infty}},\qquad D:=\|w\|_{\dot C^\alpha}=\frac{2^\alpha\|u\|_{\dot C^\alpha}}{\|u\|_{L^\infty}}\ \ge\ 2^\alpha.
\end{equation*}
Moreover, $w(z)\ge1-D|z|^\alpha$ for every $z\in\R^4$. In particular, $w\in\mathcal K_D$.
\end{lem}
\begin{proof}
By Lemma~\ref{lem:basic}(i), $u$ is continuous on $\R^4$ and vanishes outside $B$, so the point $x_1$ exists. If $w(z)\neq0$, then $|x_1+2z|<1$, hence $|z|<\frac{1+|x_1|}2<1$. Thus $w$, which belongs to $H^2(\R^4)$ as a translate and dilate of the zero extension of $u$, vanishes outside a closed ball contained in $B$. By the Notation paragraph, $w\in H^2_0(B)$. Clearly $w$ is continuous and $w(0)=1=\|w\|_{L^\infty}$. Since $\Delta w(z)=4(\Delta u)(x_1+2z)/u(x_1)$, the change of variables $y=x_1+2z$ gives
\begin{equation*}
\int_{\R^4}|\Delta w|^2\,dz=\frac{16}{u(x_1)^2}\int_{\R^4}|\Delta u(x_1+2z)|^2\,dz=\frac{16\cdot2^{-4}}{u(x_1)^2}\int_{\R^4}|\Delta u|^2\,dy=\frac{\|\Delta u\|_{L^2}^2}{\|u\|_{L^\infty}^2}.
\end{equation*}
This is the dilation invariance of $\|\Delta\cdot\|_{L^2}$ in dimension four. In the same way, $\|w\|_{\dot C^\alpha}=2^\alpha\|u\|_{\dot C^\alpha}/\|u\|_{L^\infty}$, and $D\ge2^\alpha$ by Lemma~\ref{lem:basic}(iii). Finally, for every $z\in\R^4$ we have $w(z)\ge w(0)-\|w\|_{\dot C^\alpha}|z|^\alpha=1-D|z|^\alpha$.
\end{proof}

\subsection{The key estimate}
For $K>0$ and $x\in(0,1)$ we set
\begin{equation*}
F_K(x):=D(x)^2g(x)\,\log\Big[e^3+K\sqrt{\frac{\log\big(2e+g(x)^{-1/2}\big)}{g(x)}}\,\Big],
\end{equation*}
and we denote by
\begin{equation*}
X_\alpha:=\{x\in(0,1)\ :\ Q(x)\le2^{-\alpha}\}
\end{equation*}
the set of admissible free boundary parameters. Indeed, by Lemma~\ref{lem:red} only values $D\ge2^\alpha$ occur, and then $x_D\in X_\alpha$ by \eqref{eq:Dx}. Note also that $F_K(x)=\Phi_{K,D(x)}\big(D(x)\sqrt{g(x)}\big)$, with $\Phi_{K,D}$ as in Lemma~\ref{lem:mono}. The heart of the proof is the following estimate, in which we make explicit how the various thresholds depend on the constant $K$.

\begin{prop}\label{prop:Fgeq}
There exists $K_\alpha>0$ such that
\begin{equation}\label{Fgeq}
F_{K}(x)\geq 8\pi^2\alpha\qquad \mbox{for all}\ x\in X_\alpha\ \text{and all}\ K\ge K_\alpha.
\end{equation}
\end{prop}
\begin{proof}
\emph{Small values of $x$.} Let $\bar x$ and $A$ be as in Lemma~\ref{lem:asymp}. By Lemma~\ref{lem:asymp}(b), we have
\begin{equation}\label{eq:loginvg}
\log\frac1{g(x)}=\alpha L-\log L-\log c_\alpha-\log(1+\theta(x)),\qquad|\theta(x)|\le\frac AL ,
\end{equation}
for $x\in(0,\bar x]$. We choose $x_0\in(0,\bar x]$, depending only on $\alpha$, such that $A/L\le\frac12$ and $\frac12\log\frac1{g(x)}\ge\frac\alpha4L$ for every $x\in(0,x_0]$. This is possible by \eqref{eq:loginvg}.

Fix $x\in(0,x_0]$ and $K>0$. Since $\log(e^3+Y)\ge\log Y$ for $Y>0$ and $\log(2e+g^{-1/2})\ge\frac12\log\frac1g$, we get
\begin{equation*}
\log\Big[e^3+K\sqrt{\tfrac{\log(2e+g^{-1/2})}{g}}\Big]\ge\log K+\frac12\log\frac1g+\frac12\log\Big(\frac12\log\frac1g\Big)\ge\log K+\frac12\log\frac1g+\frac12\log\Big(\frac\alpha4L\Big).
\end{equation*}
By \eqref{eq:loginvg} and $|\theta|\le\frac12$, we have $\frac12\log\frac1g\ge\frac\alpha2L-\frac12\log L-\frac12|\log c_\alpha|-\frac12\log\frac32$. The two terms $\pm\frac12\log L$ cancel, and we obtain
\begin{equation*}
\log\Big[e^3+K\sqrt{\tfrac{\log(2e+g^{-1/2})}{g}}\Big]\ge\frac\alpha2L+\log K-\beta_\alpha,\qquad \beta_\alpha:=\frac12|\log c_\alpha|+\frac12\log\frac32+\frac12\log\frac4\alpha .
\end{equation*}
Assume now that $\log K\ge\beta_\alpha+\alpha A$. Then the right-hand side is positive. Using Lemma~\ref{lem:asymp}(a) and $A/L\le\frac12$, we deduce that
\begin{align*}
F_K(x)&\ge\frac{16\pi^2}{L}\Big(1-\frac AL\Big)\Big(\frac\alpha2L+\log K-\beta_\alpha\Big)=8\pi^2\alpha+\frac{16\pi^2}{L}\Big[\Big(1-\frac AL\Big)(\log K-\beta_\alpha)-\frac{\alpha A}2\Big]\\
&\ge8\pi^2\alpha+\frac{16\pi^2}{L}\Big[\frac{\log K-\beta_\alpha}2-\frac{\alpha A}2\Big]\ge 8\pi^2\alpha .
\end{align*}
Thus \eqref{Fgeq} holds on $(0,x_0]$ for every $K\ge K_0:=\exp(\beta_\alpha+\alpha A)$. We emphasize that $x_0$ does not depend on $K$, and that we have obtained the exact inequality $F_K\ge8\pi^2\alpha$, not merely an asymptotic one.

\emph{Remaining values of $x$.} By Lemma~\ref{lem:asymp}, $Q$ is continuous on $(0,1)$ and tends to $1>2^{-\alpha}$ as $x\to1^-$. Hence there is $x_1\in(x_0,1)$ such that $X_\alpha\subset(0,x_1]$. Let $g_+:=\max_{[x_0,x_1]}g<\infty$. For $x\in X_\alpha\cap[x_0,x_1]$ we have $D(x)=1/Q(x)\ge2^\alpha\ge1$. By \eqref{eq:energy}, we get $D(x)^2g(x)\ge g(x)\ge\gamma_0:=\pi^2\alpha(\alpha+2)^2x_0^\alpha$, while $\log(2e+g^{-1/2})\ge1$. Therefore
\begin{equation*}
F_K(x)\ge\gamma_0\log\big(Kg_+^{-1/2}\big)\ge8\pi^2\alpha\qquad\text{whenever}\qquad K\ge \sqrt{g_+}\,\exp\big(8\pi^2\alpha/\gamma_0\big).
\end{equation*}
The proposition follows with $K_\alpha:=\max\big(K_0,\sqrt{g_+}\,e^{8\pi^2\alpha/\gamma_0}\big)$.
\end{proof}

\subsection{Proof of inequality \texorpdfstring{\eqref{logg}}{(1.5)}}
Let $u\in(H^2_0\cap\dot C^\alpha)(B)$ with $u\not\equiv0$, and let $w$ and $D\ge2^\alpha$ be as in Lemma~\ref{lem:red}. Set $a:=\|\Delta u\|_{L^2}/\|u\|_{L^\infty}=\|\Delta w\|_{L^2}$ and $s:=D/a$. Then $N:=N_\alpha(u)=2^{-\alpha}s\ge s/2$. Let $u^*$ be the minimizer of \eqref{min} for this value of $D$, and let $x=x_D\in X_\alpha$ be the corresponding free boundary parameter. Since $w\in\mathcal K_D$, Lemma~\ref{lem:explicit} gives $a^2\ge m(D)=D^2g(x)$.

Since $2e+\sigma\ge4$, we have $\log(2e+\frac\sigma2)\ge\log(2e+\sigma)-\log2\ge\frac12\log(2e+\sigma)$ for all $\sigma\ge0$. Hence, with $K:=K_\alpha$ and $C_\alpha:=2\sqrt2\,K_\alpha$,
\begin{equation*}
\log\Big(e^3+C_\alpha N\sqrt{\log(2e+N)}\Big)\ge\log\Big(e^3+\tfrac{C_\alpha}{2}s\sqrt{\log(2e+\tfrac s2)}\Big)\ge\log\Big(e^3+K s\sqrt{\log(2e+s)}\Big).
\end{equation*}
We multiply by $a^2$ and apply Lemma~\ref{lem:mono} with $t=a\ge\sqrt{m(D)}=D\sqrt{g(x)}$. By Proposition~\ref{prop:Fgeq}, we obtain
\begin{equation*}
\frac{\|\Delta u\|_{L^2}^2}{\|u\|_{L^\infty}^2}\log\Big(e^3+C_\alpha N\sqrt{\log(2e+N)}\Big)\ge\Phi_{K,D}(a)\ge\Phi_{K,D}\big(D\sqrt{g(x)}\big)=F_{K}(x)\ge8\pi^2\alpha .
\end{equation*}
This is exactly \eqref{logg}.\qed

\subsection{\texorpdfstring{Optimality of the constant $\frac{1}{8\pi^2\alpha}$}{Optimality of the constant}}
Let $\varepsilon\in(0,1)$ be small and set $\Lambda:=\log\frac1\varepsilon$. We use the following Adams-type concentrating profile:
\begin{equation*}
v_{\varepsilon}(y):=\left\{\begin{array}{ll}
\sqrt{\frac{\Lambda}{32\pi^2}}-\frac{|y|^2}{\sqrt{8\pi^2\varepsilon\Lambda}}+\frac{1}{\sqrt{8\pi^2\Lambda}}&\mbox{if}\ |y|\leq\varepsilon^{1/4},\\[4pt]
\frac{1}{\sqrt{2\pi^2\Lambda}}\log\frac{1}{|y|}&\mbox{if}\ \varepsilon^{1/4}\leq |y|\leq 1,\\[4pt]
\frac{f(|y|)}{\sqrt{2\pi^2\Lambda}},\quad f(r):=(1-r)(r-2)^2,&\mbox{if}\ 1<|y|<2,\\[4pt]
0&\mbox{if}\ |y|\geq 2,
\end{array}\right.
\end{equation*}
and we set $u_{\varepsilon}(y):=v_{\varepsilon}(2y)$. The values and the first radial derivatives of the four pieces agree at $|y|=\varepsilon^{1/4}$, $|y|=1$ and $|y|=2$. For the last two spheres we use $f(1)=f(2)=f'(2)=0$ and $f'(1)=-1$. Hence $v_\varepsilon\in C^1(\R^4)\cap H^2(\R^4)$ vanishes outside $\overline{B_2}$, and the Notation paragraph gives $u_\varepsilon\in H^2_{0,rad}(B)$.

The profile of $v_\varepsilon$ is decreasing on $[0,1]$, and $|v_\varepsilon|\le\frac{4/27}{\sqrt{2\pi^2\Lambda}}$ on $\{1<|y|<2\}$. Hence, for small $\varepsilon$,
\begin{equation*}
\|u_{\varepsilon}\|_{L^{\infty}}^2=v_\varepsilon(0)^2=\Big(\sqrt{\tfrac{\Lambda}{32\pi^2}}+\tfrac{1}{\sqrt{8\pi^2\Lambda}}\Big)^2\ge\frac{\Lambda}{32\pi^2}.
\end{equation*}
By the dilation invariance used in Lemma~\ref{lem:red}, $\|\Delta u_\varepsilon\|_{L^2}^2=\|\Delta v_\varepsilon\|_{L^2}^2$. On the middle piece we have $\Delta\log\frac1{|y|}=-2|y|^{-2}$, whose contribution is $\frac{4}{2\pi^2\Lambda}\cdot2\pi^2\int_{\varepsilon^{1/4}}^1\frac{dr}r=1$. The inner piece contributes $\frac{64}{8\pi^2\varepsilon\Lambda}\cdot\frac{\pi^2\varepsilon}2=\frac4\Lambda$, and the outer one contributes $\frac{\kappa_0}{\Lambda}$, with $\kappa_0:=\int_1^2|f''+\frac3rf'|^2r^3\,dr$. Thus
\begin{equation*}
\|\Delta u_\varepsilon\|_{L^2}^2=1+\frac{4+\kappa_0}{\Lambda}.
\end{equation*}
We now estimate the H\"older seminorm. On $[0,1]$, the absolute value of the radial derivative of $v_\varepsilon$ is maximal at $r=\varepsilon^{1/4}$, where it equals $\frac{\varepsilon^{-1/4}}{\pi\sqrt{2\Lambda}}$. On $(1,2)$, it is bounded by $\frac1{\sqrt{2\pi^2\Lambda}}$, since $|f'(r)|=|(r-2)(4-3r)|\le1$ there. Hence $\|u_\varepsilon\|_{\mathrm{Lip}}\le\frac{2\varepsilon^{-1/4}}{\pi\sqrt{2\Lambda}}$ for small $\varepsilon$. For any bounded Lipschitz function $u$ we have
\begin{equation*}
|u(y)-u(y')|\le\min\big(2\|u\|_{L^\infty},\|u\|_{\mathrm{Lip}}|y-y'|\big)\le(2\|u\|_{L^\infty})^{1-\alpha}\|u\|_{\mathrm{Lip}}^\alpha|y-y'|^\alpha .
\end{equation*}
Since $\|u_\varepsilon\|_{L^\infty}\lesssim\Lambda^{1/2}$ and $\|\Delta u_\varepsilon\|_{L^2}\ge1$, we deduce that
\begin{equation*}
\|u_{\varepsilon}\|_{\dot C^{\alpha}}\lesssim\Lambda^{\frac{1-\alpha}2}\,\varepsilon^{-\frac\alpha4}\Lambda^{-\frac\alpha2}=\Lambda^{\frac12-\alpha}\,\varepsilon^{-\frac\alpha4}\qquad\text{and}\qquad N_\alpha(u_\varepsilon)\lesssim\Lambda^{\frac12-\alpha}\,\varepsilon^{-\frac\alpha4}.
\end{equation*}
In particular $u_\varepsilon\in(H^2_0\cap\dot C^\alpha)(B)$. Fix $C_0>0$. Then $\log N_\alpha(u_\varepsilon)\le\frac\alpha4\Lambda+|\frac12-\alpha|\log\Lambda+C$ and $\log\log(2e+N_\alpha(u_\varepsilon))\le\log\Lambda+C$. We infer that
\begin{equation*}
\log\Big[e^3+C_0N_{\alpha}(u_\varepsilon)\sqrt{\log(2e+N_\alpha(u_\varepsilon))}\Big]\le\frac\alpha4\Lambda+C'\log\Lambda
\end{equation*}
for small $\varepsilon$, where $C'$ depends only on $\alpha$ and $C_0$. Consequently,
\begin{align*}
R_\varepsilon&:=\frac{\|\Delta u_{\varepsilon}\|_{L^2}^2}{\|u_{\varepsilon}\|_{L^{\infty}}^2}\log\Big[e^3+C_0N_{\alpha}(u_\varepsilon)\sqrt{\log(2e+N_\alpha(u_\varepsilon))}\Big]\\
&\le\Big(1+\frac{4+\kappa_0}{\Lambda}\Big)\frac{32\pi^2}{\Lambda}\Big(\frac\alpha4\Lambda+C'\log\Lambda\Big)\xrightarrow[\varepsilon\to0]{}8\pi^2\alpha .
\end{align*}
Let $\lambda<\frac1{8\pi^2\alpha}$. Then $\limsup_{\varepsilon\to0}\lambda R_\varepsilon\le8\pi^2\alpha\lambda<1$, so that $\lambda R_\varepsilon<1$ for $\varepsilon$ small. In other words,
\begin{equation*}
\|u_\varepsilon\|_{L^\infty}^2>\lambda\|\Delta u_\varepsilon\|_{L^2}^2\log\big[e^3+C_0N_{\alpha}(u_\varepsilon)\sqrt{\log(2e+N_\alpha(u_\varepsilon))}\big].
\end{equation*}
Since $C_0>0$ is arbitrary, the constant $\frac1{8\pi^2\alpha}$ in \eqref{logg} cannot be decreased. This completes the proof of Theorem~\ref{t1}.\qed

%%%%%%%%%%%%%%%%%%%%%%%%%%%%%%%%%%%%%%%%%%%%%%%%%%%%%%%%%%%%%%%%%%%%%%%%%%%%%%%%%%%%%%%%%%%%%%%%%%%%%%%%%%%%%%%%%%
\section{Proof of Theorem \ref{t2}}
\label{sec:prooft2}
%%%%%%%%%%%%%%%%%%%%%%%%%%%%%%%%%%%%%%%%%%%%%%%%%%%%%%%%%%%%%%%%%%%%%%%%%%%
\subsection{\texorpdfstring{The estimate for $\lambda>\frac1{8\pi^2\alpha}$}{The estimate above the threshold}}
Let $\lambda>\frac{1}{8\pi^2\alpha}$. We keep the notation of Sections~\ref{sec:explicit} and \ref{sec:pfT1}. For $C\ge1$ and $x\in(0,1)$ we put
\begin{equation*}
H_C(x):=D(x)^2g(x)\log\Big(C+\frac{1}{2\sqrt{g(x)}}\Big).
\end{equation*}

We first show that there is $C_\lambda\ge e$ such that $\lambda H_{C_\lambda}(x)\ge1$ for every $x\in X_\alpha$. Let $x_0$ and $A$ be as in the proof of Proposition~\ref{prop:Fgeq}. For $x\in(0,x_0]$ and $C\ge1$, \eqref{eq:loginvg} gives
\begin{equation*}
\log\Big(C+\frac1{2\sqrt g}\Big)\ge\frac12\log\frac1g-\log2\ge\frac\alpha2L-\frac12\log L-\beta'_\alpha,\qquad\beta'_\alpha:=\frac12|\log c_\alpha|+\frac12\log\frac32+\log2 .
\end{equation*}
By Lemma~\ref{lem:asymp}(a), it follows that, as soon as the right-hand side is nonnegative,
\begin{equation*}
\lambda H_C(x)\ge\lambda\,\frac{16\pi^2}L\Big(1-\frac AL\Big)\Big(\frac\alpha2L-\frac12\log L-\beta'_\alpha\Big)=8\pi^2\alpha\lambda\Big(1-\frac AL\Big)\Big(1-\frac{\log L+2\beta'_\alpha}{\alpha L}\Big).
\end{equation*}
The right-hand side does not depend on $C$, and it tends to $8\pi^2\alpha\lambda>1$ as $x\to0^+$. Hence there exists $x_\lambda\in(0,x_0]$, independent of $C$, such that $\lambda H_C(x)\ge1$ for all $x\in(0,x_\lambda]$ and all $C\ge1$.

Let now $x_1$ be as in the proof of Proposition~\ref{prop:Fgeq}. Every $x\in X_\alpha\cap[x_\lambda,1)$ lies in $[x_\lambda,x_1]$ and satisfies $D(x)\ge1$. By \eqref{eq:energy}, we get $H_C(x)\ge g(x)\log C\ge\gamma_\lambda\log C$, with $\gamma_\lambda:=\pi^2\alpha(\alpha+2)^2x_\lambda^\alpha$. Therefore the choice
\begin{equation*}
C_\lambda:=\max\Big(e,\ \exp\big(\tfrac{1}{\lambda\gamma_\lambda}\big)\Big)
\end{equation*}
gives $\lambda H_{C_\lambda}(x)\ge1$ for every $x\in X_\alpha$.

Let now $u\in(H^2_0\cap\dot C^\alpha)(B)$ with $u\not\equiv0$, and let $w$, $D\ge2^\alpha$ and $a=\|\Delta u\|_{L^2}/\|u\|_{L^\infty}$ be as in Lemma~\ref{lem:red} and in the proof of \eqref{logg}. Then $N_\alpha(u)=2^{-\alpha}D/a\ge D/(2a)$ and $a\ge\sqrt{m(D)}=D\sqrt{g(x_D)}$. Applying Lemma~\ref{lem:mono} with $D/2$ in place of $D$, we obtain
\begin{equation*}
\frac{\|\Delta u\|_{L^2}^2}{\|u\|_{L^\infty}^2}\log\big(C_\lambda+N_\alpha(u)\big)\ge a^2\log\Big(C_\lambda+\frac{D/2}{a}\Big)\ge m(D)\log\Big(C_\lambda+\frac{D/2}{\sqrt{m(D)}}\Big)=H_{C_\lambda}(x_D)\ge\frac1\lambda,
\end{equation*}
which is \eqref{log}.

\subsection{\texorpdfstring{Failure of \eqref{log} for $\lambda=\frac{1}{8\pi^2\alpha}$}{Failure of (1.6) at the threshold}}
Let $C>0$. Since $\log(C+N)$ increases with $C$, it suffices to treat the case $C\ge1$. We use the explicit functions $u_x$ of Lemma~\ref{lem:family} with $x\in(0,\bar x]$. By that lemma, $u_x\in(H^2_0\cap\dot C^\alpha)(B)$ is radial, $\|u_x\|_{L^\infty}=1$ and $\|\Delta u_x\|_{L^2}^2=D(x)^2g(x)$. Moreover, \eqref{bound} gives
\begin{equation*}
N_\alpha(u_x)=\frac{\|u_x\|_{\dot C^\alpha}}{D(x)\sqrt{g(x)}}\le\frac{5}{\sqrt{g(x)}} .
\end{equation*}
Since $g(x)\to0$ as $x\to0^+$, we have $5g(x)^{-1/2}\ge C$ for $x$ small, and therefore
\begin{equation*}
\log(C+N_\alpha(u_x))\le\log\big(10\,g(x)^{-1/2}\big)=\log10+\tfrac12\log\tfrac1{g(x)}.
\end{equation*}
Using \eqref{eq:loginvg} with $|\theta|\le\frac12$ and Lemma~\ref{lem:asymp}(a), we obtain, for $x$ small,
\begin{align*}
\|\Delta u_x\|_{L^2}^2\log\big(C+N_{\alpha}(u_x)\big)&\le\frac{16\pi^2}{L}\Big(1+\frac AL\Big)\Big(\frac\alpha2L-\frac12\log L+\beta''_\alpha\Big)\\
&=8\pi^2\alpha+\frac{16\pi^2}{L}\Big[-\frac12\log L+\beta''_\alpha+\frac{\alpha A}2+\frac AL\Big(\beta''_\alpha-\frac12\log L\Big)\Big],
\end{align*}
with $\beta''_\alpha:=\log10+\frac12|\log c_\alpha|+\frac12\log2$. The bracket tends to $-\infty$ as $x\to0^+$. Hence, for $x$ small enough,
\begin{equation*}
\frac{1}{8\pi^2\alpha}\|\Delta u_x\|_{L^2}^2\log\big(C+N_{\alpha}(u_x)\big)<1=\|u_x\|_{L^\infty}^2 .
\end{equation*}
Consequently, \eqref{log} fails for $\lambda=\frac1{8\pi^2\alpha}$, whatever the value of the constant $C$. This argument requires no limiting procedure in $C$. The loss comes from the term $-\frac{8\pi^2\log L}{L}$. A single logarithm cannot compensate this term, and this is precisely what the double logarithm in \eqref{logg} absorbs. This completes the proof of Theorem~\ref{t2}.\qed

%%%%%%%%%%%%%%%%%%%%%%%%%%%%%%%%%%%%%%%%%%%%%%%%%%%%%%%%%%%%%%%%%%%%%%%%%%%%%%%%%%%%%%%%%%%%%%%%%%%%%%%%%%%%%%%%%%%%%%%%%%%%%%%%%%%%%%
\section{Balls and the whole space}
\label{sec:wholespace}
%%%%%%%%%%%%%%%%%%%%%%%%%%%%%%%%%%%%%%%%%%%%%%%%%%%%%%%%%%%%%%%%%%%%%%%%%%%%%%%%%%%%%%%%%%%%%%%%%%%%%%%%%%%%%%%%%%
Theorems \ref{t1} and \ref{t2} are stated in the unit ball, but they extend to all balls by scaling. Let $R>0$ and $u\in(H^2_0\cap\dot C^\alpha)(B_R)$, and set $u_R(y):=u(Ry)$. Then $u_R\in(H^2_0\cap\dot C^\alpha)(B)$, $\|u_R\|_{L^\infty}=\|u\|_{L^\infty}$, $\|\Delta u_R\|_{L^2}=\|\Delta u\|_{L^2}$ by dilation invariance, and $\|u_R\|_{\dot C^\alpha}=R^\alpha\|u\|_{\dot C^\alpha}$. Hence $N_\alpha(u_R)=R^\alpha N_\alpha(u)$. Applying Theorem~\ref{t1} to $u_R$, we get
\begin{equation*}
\|u\|_{L^{\infty}(B_R)}^2\leq\frac1{8\pi^2\alpha}\|\Delta u\|_{L^2(B_R)}^2\log\Big[e^3+C_\alpha R^\alpha N_\alpha(u)\sqrt{\log(2e+R^{\alpha}N_{\alpha}(u))}\Big].
\end{equation*}
In the same way, Theorem \ref{t2} yields the following statement.
\begin{cor}{\bf (Log estimate on $B_R$)}\label{cr1}
Let $\alpha\in (0,1)$. For any $\lambda>\frac{1}{8\pi^2\alpha}$ there exists $C_{\lambda}\ge1$ such that, for any $R>0$ and any function $u\in(H^2_0\cap \dot C^{\alpha})(B_R)$ with $u\not\equiv0$,
\begin{equation}\label{llogR}
\|u\|_{L^{\infty}}^2\leq\lambda\|\Delta u\|_{L^2}^2\log\Big(C_{\lambda}+R^{\alpha}N_{\alpha}(u)\Big).
\end{equation}
\end{cor}
We now prove Corollary~\ref{cor2}. For the convenience of the reader, we restate it.
\begin{cor}{\bf (Global Log estimate)}\label{cr2}
Let $\alpha\in (0,1)$. For any $\lambda>\frac{1}{8\pi^2\alpha}$ there exists $C_{\lambda}>0$ such that, for any $\mu\in (0,1]$ and any function $u\in(H^2\cap C^{\alpha})(\R^4)$ with $u\not\equiv0$,
\begin{equation}\label{lllogR}
\|u\|_{L^{\infty}}^2\leq\lambda\| u\|_{\mu}^2\log\Big(C_{\lambda}+\frac{8^{\alpha}\mu^{-\alpha}\|u\|_{C^{\alpha}}}{\|u\|_{\mu}}\Big),
\end{equation}
where $\|u\|_{\mu}^2=(1+3\mu)\|\Delta u\|_{L^2}^2+3\mu\|u\|_{H^1}^2$.
\end{cor}
\begin{proof}
Let $\lambda>\frac{1}{8\pi^2\alpha}$, $\mu\in (0,1]$ and $u\in(H^2\cap C^{\alpha})(\R^4)$ with $u\not\equiv0$, and let $C_\lambda\ge1$ be given by Corollary~\ref{cr1}. Since $u$ is uniformly continuous and belongs to $L^2(\R^4)$, it tends to $0$ at infinity. Otherwise there would be $\varepsilon>0$ and points $x_n\to\infty$ with $|u(x_n)|\ge\varepsilon$, and by uniform continuity $|u|\ge\varepsilon/2$ on infinitely many disjoint balls of a fixed radius. Hence $|u|$ attains its maximum. All quantities in \eqref{lllogR} are invariant under translations, so we may assume that $\|u\|_{L^{\infty}}=|u(0)|$. No symmetry of $u$ is needed here, because Corollary~\ref{cr1} holds for non-radial functions.

We first fix a radial function $\phi\in C_0^{\infty}(B_4)$ such that $0\leq \phi\leq 1$, $\phi(0)=1$, $|\nabla \phi|\leq 1$ and $|\Delta \phi|\leq 1$. Such a function can be constructed as follows. Let $\kappa=\frac9{10}$, $m>0$ with $m^2=\frac{4}{\kappa(2-\sqrt2)}$, and $M:=2^{1/4}m$. Let $\phi_\flat$ be the radial $C^{1,1}$ function such that $\phi_\flat(0)=1$, $\Delta\phi_\flat=-\kappa$ on $B_m$, $\Delta\phi_\flat=\kappa$ on $B_{M}\setminus B_m$, and $\phi_\flat=0$ outside $B_{M}$. Integrating the radial equation $(r^3\phi_\flat')'=r^3\Delta\phi_\flat$, we find $\phi_\flat'(r)=-\frac{\kappa r}4$ on $[0,m]$ and $\phi_\flat'(r)=\frac{\kappa(r^4-2m^4)}{4r^3}$ on $[m,M]$. Hence $\phi_\flat'(M)=0$, the function $\phi_\flat$ is nonincreasing in $|y|$, and
\begin{equation*}
\phi_\flat(0)-\phi_\flat(M)=\frac{\kappa m^2}8+\frac\kappa4\int_m^M\Big(\frac{2m^4}{r^3}-r\Big)dr=\frac{\kappa m^2}{8}+\frac{\kappa m^2}{4}\Big(\frac32-\sqrt2\Big)=\frac{\kappa(2-\sqrt2)m^2}{4}=1,
\end{equation*}
so that $\phi_\flat(M)=0$. Moreover $|\nabla\phi_\flat|\le\kappa m/4<0.62$ and $M<3.3$. Let $\rho_\delta$ be a radial, radially nonincreasing mollifier supported in $B_\delta$. The convolution $\phi_\flat*\rho_\delta$ is radially nonincreasing, as a convolution of two radially nonincreasing functions (see \cite[Chapter~3]{LiLo01}), and $(\phi_\flat*\rho_\delta)(0)\to1$ as $\delta\to0$. Moreover $|\Delta(\phi_\flat*\rho_\delta)|\le\kappa$ and $|\nabla(\phi_\flat*\rho_\delta)|<0.62$. For $\delta$ small, $(\phi_\flat*\rho_\delta)(0)\ge\kappa$, and the function $\phi:=(\phi_\flat*\rho_\delta)/(\phi_\flat*\rho_\delta)(0)$ has all the required properties.

Let $\phi_{\mu}:=\phi(\tfrac{\mu}{2}\,\cdot)$ and $u_{\mu}:=\phi_{\mu}u$. Then $u_\mu\in H^2(\R^4)$ vanishes outside a closed ball contained in $B_{8/\mu}$, so that $u_\mu\in H^2_0(B_{8/\mu})$. Moreover $|\nabla\phi_\mu|\le\frac\mu2$, $|\Delta\phi_\mu|\le\frac{\mu^2}4$, and $\|u_{\mu}\|_{L^{\infty}}=|u_\mu(0)|=\|u\|_{L^{\infty}}$. Since $0\le\phi_\mu\le1$, we have $|\phi_\mu(y)-\phi_\mu(y')|\le\min\big(1,\tfrac\mu2|y-y'|\big)\le(\tfrac\mu2)^\alpha|y-y'|^\alpha$. It follows that
\begin{equation*}
|u_\mu(y)-u_\mu(y')|\le\phi_\mu(y)|u(y)-u(y')|+|u(y')|\,|\phi_\mu(y)-\phi_\mu(y')|\le\big(\|u\|_{\dot C^\alpha}+\|u\|_{L^\infty}\big)|y-y'|^\alpha,
\end{equation*}
that is, $\|u_{\mu}\|_{\dot C^{\alpha}}\leq \|u\|_{C^{\alpha}}$. We apply Corollary \ref{cr1} with $R=8/\mu$. Since $N\mapsto\log(C_\lambda+N)$ is increasing, we obtain
\begin{equation*}
\|u\|_{L^{\infty}}^2\leq\lambda\|\Delta u_{\mu}\|_{L^2}^2\log\Big(C_{\lambda}+\frac{8^{\alpha}\mu^{-\alpha}\|u\|_{ C^{\alpha}}}{\|\Delta u_{\mu}\|_{L^2}}\Big).
\end{equation*}
It remains to compare $\|\Delta u_\mu\|_{L^2}$ with $\|u\|_\mu$. Since $\Delta u_\mu=\phi_\mu\Delta u+2\nabla\phi_\mu\cdot\nabla u+u\Delta\phi_\mu$, we have
\begin{align*}
\|\Delta u_{\mu}\|_{L^2}^2
&=\|(\Delta \phi_{\mu})u\|_{L^2}^2+\|\phi_{\mu}\Delta u\|_{L^2}^2+4\|\nabla \phi_{\mu}\cdot\nabla u\|_{L^2}^2\\
&\qquad+2\int(\Delta \phi_{\mu})\phi_{\mu}u\Delta u+4\int(\Delta \phi_{\mu})u\nabla\phi_{\mu}\cdot\nabla u+4\int \phi_{\mu}(\Delta u)\nabla \phi_{\mu}\cdot\nabla u\\
&\leq\tfrac{\mu^4}{16}\|u\|_{L^2}^2+\|\Delta u\|_{L^2}^2+\mu^2\|\nabla u\|_{L^2}^2+2|(I)|+4|(II)|+4|(III)|.
\end{align*}
By the Cauchy--Schwarz and Young inequalities, the three cross terms satisfy
\begin{gather*}
|(I)|=\Big|\int(\Delta \phi_{\mu})\phi_{\mu}u\Delta u\Big|\leq \tfrac{\mu^2}{8}(\|u\|_{L^2}^2+\|\Delta u\|_{L^2}^2),\\
|(II)|=\Big|\int(\Delta \phi_{\mu})u\nabla\phi_{\mu}\cdot\nabla u\Big|\leq \tfrac{\mu^3}{16}(\|u\|_{L^2}^2+\|\nabla u\|_{L^2}^2),\\
|(III)|=\Big|\int \phi_{\mu}(\Delta u)\nabla \phi_{\mu}\cdot\nabla u\Big|\leq \tfrac{\mu}{4}(\|\Delta u\|_{L^2}^2+\|\nabla u\|_{L^2}^2).
\end{gather*}
Collecting terms and using $\mu\le1$, we see that the coefficients of $\|u\|_{L^2}^2$, $\|\nabla u\|_{L^2}^2$ and $\|\Delta u\|_{L^2}^2$ are bounded by $\frac{9}{16}\mu$, $\frac94\mu$ and $1+\frac54\mu$, respectively. Hence $\|\Delta u_\mu\|_{L^2}\le\|u\|_\mu$. Finally, $t\mapsto t^2\log(C_{\lambda}+c/t)$ is nondecreasing for every $c>0$ by Lemma~\ref{lem:mono}, since $C_\lambda\ge1$. This completes the proof.
\end{proof}

\begin{proof}[Proof of Corollary~\ref{cor:L2}]
With the notation of Section~\ref{sec:LPP}, we write $u=\Delta_{-1}u+v$, where $\Delta_{-1}u=\widetilde h*u$. Then
\begin{equation*}
\|\Delta_{-1}u\|_{L^\infty}\le\|\widetilde h\|_{L^2}\|u\|_{L^2}=:C_0\|u\|_{L^2}\qquad\text{and}\qquad\|v\|_{C^\alpha}\le\big(1+\|\widetilde h\|_{L^1}\big)\|u\|_{C^\alpha}.
\end{equation*}
The constant $C_0$ depends only on the function $\chi$, which was fixed once and for all. Moreover, $\mathcal Fv=(1-\chi)\mathcal Fu$, where $0\le1-\chi\le1$ and $1-\chi=0$ on $B_{3/4}$. By Plancherel's theorem, we get $\|\Delta v\|_{L^2}\le\|\Delta u\|_{L^2}$, $\|\nabla v\|_{L^2}\le\frac43\|\Delta v\|_{L^2}$ and $\|v\|_{L^2}\le\frac{16}9\|\Delta v\|_{L^2}$. If $v\equiv0$, there is nothing to prove. Otherwise, let $\lambda_1\in(\frac1{8\pi^2\alpha},\lambda)$ and $\mu_1\in(0,1]$ be such that $\lambda_1\kappa_1\le\lambda$, where $\kappa_1:=1+3\mu_1\big(1+\frac{256}{81}+\frac{16}9\big)$. Then $\|v\|_{\mu_1}\le\sqrt{\kappa_1}\|\Delta v\|_{L^2}\le\sqrt{\kappa_1}\|\Delta u\|_{L^2}$. We apply Corollary~\ref{cr2} to $v$ with $\lambda_1$ and $\mu_1$. Since $t\mapsto t\sqrt{\log(C_{\lambda_1}+c/t)}$ is nondecreasing by Lemma~\ref{lem:mono}, and since $\kappa_1\ge1$, we obtain
\begin{equation*}
\|v\|_{L^\infty}\le\sqrt{\kappa_1}\|\Delta u\|_{L^2}\sqrt{\lambda_1\log\Big(C_{\lambda_1}+\frac{c_1\|u\|_{C^\alpha}}{\sqrt{\kappa_1}\|\Delta u\|_{L^2}}\Big)}\le\|\Delta u\|_{L^2}\sqrt{\lambda\log\Big(C_{\lambda_1}+\frac{c_1\|u\|_{C^\alpha}}{\|\Delta u\|_{L^2}}\Big)},
\end{equation*}
with $c_1:=8^\alpha\mu_1^{-\alpha}(1+\|\widetilde h\|_{L^1})$. The claim follows with $C_\lambda:=\max(C_{\lambda_1},c_1)$.
\end{proof}

\begin{rem}\rm\label{rem:wholeloglog}
The double logarithmic inequality \eqref{logg} extends in the same way to the whole space, with the sharp constant $\frac1{8\pi^2\alpha}$. Indeed, we repeat the proof of Corollary~\ref{cr2}, using the rescaled form of Theorem~\ref{t1} on $B_{8/\mu}$ instead of Corollary~\ref{cr1}, and the monotonicity of $\Phi_{K,D}$ from Lemma~\ref{lem:mono} with $K=C_\alpha$ and $D=8^\alpha\mu^{-\alpha}\|u\|_{C^\alpha}$. We obtain, for every $\mu\in(0,1]$ and every $u\in(H^2\cap C^\alpha)(\R^4)$ with $u\not\equiv0$,
\begin{equation*}
\|u\|_{L^\infty}^2\le\frac{1}{8\pi^2\alpha}\|u\|_\mu^2\log\Big(e^3+C_\alpha Y\sqrt{\log(2e+Y)}\Big),\qquad Y:=\frac{8^\alpha\mu^{-\alpha}\|u\|_{C^\alpha}}{\|u\|_\mu}.
\end{equation*}
The constant $\frac1{8\pi^2\alpha}$ cannot be decreased here either. Take $u=u_\varepsilon$ as in Section~\ref{sec:pfT1} and $\mu=\Lambda^{-1}$. A direct computation on the three pieces of $v_\varepsilon$ gives $\|u_\varepsilon\|_{H^1}^2=O(\Lambda^{-1})$, so that $\|u_\varepsilon\|_\mu^2=\|\Delta u_\varepsilon\|_{L^2}^2\big(1+O(\Lambda^{-1})\big)$. Moreover $\|u_\varepsilon\|_{C^\alpha}\lesssim\Lambda^{1/2}+\Lambda^{\frac12-\alpha}\varepsilon^{-\alpha/4}$ and $\mu^{-\alpha}=\Lambda^\alpha$, hence $\log\big(e^3+C_\alpha Y\sqrt{\log(2e+Y)}\big)\le\frac\alpha4\Lambda+C'\log\Lambda$. The argument at the end of Section~\ref{sec:pfT1} then applies without change.
\end{rem}

%@@@@@@@@@@@@@@@@@@@@@@@@@@@@@@@@@@@@%@@@@@@@@@@@@@@@@@@@@@@@@@@@@@@@@@@@@%@@@@@@@@@@@@@@

\section*{Declarations}
\noindent\textbf{Funding.} No funding was received for conducting this study.

\noindent\textbf{Conflict of interest.} The authors declare that they have no conflict of interest.

\noindent\textbf{Data availability.} Data sharing is not applicable to this article, as no datasets were generated or analysed.

%@@@@@@@@@@@@@@@@@@@@@@@@@@@@@@@@@@@@%@@@@@@@@@@@@@@@@@@@@@@@@@@@@@@@@@@@@%@@@@@@@@@@@@@@

\end{document}